\documentclass[]{interact}

\usepackage{tikz}
\usetikzlibrary{positioning,calc,arrows.meta,shapes,arrows,fit}

\usepackage{graphicx}
\usepackage{epstopdf}

\usepackage[caption=false]{subfig}

\usepackage{caption}
\usepackage{multirow}
\usepackage{amsmath,amssymb,amsfonts}
\usepackage{amsthm}
\usepackage{mathrsfs}
\usepackage[title]{appendix}
\usepackage{xcolor}
\usepackage{textcomp}
\usepackage{manyfoot}
\usepackage{booktabs}
\usepackage{algorithm}
\usepackage{algorithmicx}
\usepackage{algpseudocode}
\usepackage{listings}
\usepackage{CJK}
\usepackage{centernot}
\usepackage{indentfirst}
\usepackage{cases}
\usepackage{longtable}
\usepackage{array}
\usepackage{hyperref}
\usepackage{colortbl}
\usepackage{nicematrix}
\usepackage{blkarray}
\usepackage{float} 

\definecolor{myDarkGreen}{rgb}{0,0.5,0}
 \usepackage{blkarray}

\usepackage{epstopdf}% To incorporate .eps illustrations using PDFLaTeX, etc.
\usepackage[caption=false]{subfig}% Support for small, `sub' figures and tables

\usepackage[numbers,sort&compress]{natbib}% Citation support using natbib.sty
\bibpunct[, ]{[}{]}{,}{n}{,}{,}% Citation support using natbib.sty
\renewcommand\bibfont{\fontsize{10}{12}\selectfont}% Bibliography support using natbib.sty
\makeatletter% @ becomes a letter
\def\NAT@def@citea{\def\@citea{\NAT@separator}}% Suppress spaces between citations using natbib.sty
\makeatother% @ becomes a symbol again

\theoremstyle{plain}% Theorem-like structures provided by amsthm.sty
\newtheorem{theorem}{Theorem}[section]

\newtheorem{proposition}[theorem]{Proposition}

\theoremstyle{definition}
\newtheorem{definition}[theorem]{Definition}
\newtheorem{example}[theorem]{Example}

\theoremstyle{remark}
\newtheorem{remark}{Remark}

\begin{document}

	\title{On Constraint Qualifications for MPECs with Applications to Bilevel Hyperparameter Optimization for Machine Learning}
	
	\author{
		\name{Jiani Li\textsuperscript{a}\thanks{Email: lijiani\_qaz@163.com}, 
			Qingna Li\textsuperscript{a,b}\thanks{Email: qnl@bit.edu.cn. Corresponding author. }
            and 
            Alain Zemkoho\textsuperscript{c}\thanks{Email: a.b.zemkoho@soton.ac.uk}
            }
		\affil{\textsuperscript{a}School of Mathematics and Statistics, Beijing Institute of Technology, Beijing, China;
        \textsuperscript{b}Beijing Key Laboratory on MCAACI/Key Laboratory of Mathematical Theory and Computation in Information Security, Beijing Institute of Technology, Beijing, China;
        \textsuperscript{c}School of Mathematical Sciences, University of Southampton, Southampton, SO17 1BJ, UK.
        }
	}
	
	\maketitle
    
	\begin{abstract}
        Constraint qualifications for a Mathematical Program with Equilibrium Constraints (MPEC) are essential for analyzing stationarity properties and establishing convergence results.  
        In this paper, we explore several classical MPEC constraint qualifications   
        and clarify the relationships among them. 
        We subsequently examine the behavior of these constraint qualifications in the context of a specific MPEC derived from bilevel hyperparameter optimization (BHO) for L1-loss support vector classification.
        In particular, for such an MPEC, we show that calmness, MPEC-CPLD and MPEC-RCRCQ hold automatically, while under certain assumptions, MPEC-LICQ, MPEC-MFCQ-T and MPEC-MFCQ-R also hold.
        
	\end{abstract}
	
	
	\section{Introduction}
    A bilevel optimization problem (BLO) is a hierarchical optimization problem in which the feasible region of the upper-level problem depends on the solution set of the lower-level problem \cite{Bard,SDempe1,DempeZemkohoBook}. 
It has been extensively studied because of its wide applications across various fields of science and engineering \cite{SDempe2,Sinha,Wogrin}, especially in the fast-developing area of machine learning \cite{hyper2,architecture2,rein2,image2}.

    In general, a BLO is highly complicated and computationally challenging to solve due to its nonconvexity and nondifferentiability \cite{SDempe3}. A commonly used approach for solving a BLO is to replace the lower-level problem with its Karush–Kuhn–Tucker (KKT) conditions, thereby transforming it into a mathematical program with equilibrium constraints (MPEC) \cite{DempeZemkoho2013,DempeZemkoho2012,LuoZQMPEC}. 
    Due to the presence of complementarity constraints, many commonly used constraint qualifications (CQs) in nonlinear programming (NLP), such as the linear independence constraint qualification (LICQ) and the Mangasarian-Fromovitz constraint qualification (MFCQ) fail to hold at all feasible points of an MPEC  \cite{MLF1}. 
    When CQs fail, the KKT conditions may not be valid at a minimizer, and the convergence assumptions for most standard methods for constrained optimization may not be satisfied \cite{ramos}.

    For these reasons, establishing appropriate constraint qualifications tailored to the complementarity structure in MPECs is crucial, in part to characterize local optimal solutions. Specialized CQs that capture the structure of the problem have been widely studied in the literature; see, e.g., \cite{LuoZQMPEC,Pang1999}.
    For example,  the \emph{tightened} NLP–based extension of MFCQ (MPEC-MFCQ-T) was proposed by Scheel and Scholtes \cite{Scheel2000}, and the \emph{relaxed} NLP–based MFCQ (MPEC-MFCQ-R) was introduced by Ralph and Wright \cite{Ralph04}. Ye \cite{JaneYe05} developed other extensions of MFCQ for MPECs such as MPEC piecewise MFCQ. Flegel and Kanzow \cite{MLF1,MFL2} explored MPEC-specific adaptations of the Abadie constraint qualification (MPEC-ACQ) and the Guignard constraint qualification (MPEC-GCQ). Further developments include the work by Henrion and Outrata \cite{Henrion}, who studied conditions for detecting the calmness property in constraint set-valued mappings. 
    
    {  Several other CQs were also proposed that are weaker than classical LICQ and MFCQ but still sufficient to guarantee stationarity and local error bounds~\cite{Alberto19,CL13}. Guo, Lin, and Zhang introduced the relaxed constant rank constraint qualification (RCRCQ), the relaxed constant positive linear dependence condition (RCPLD), and the constant rank of the subspace component condition (CRSC). They showed that these conditions yield local error bounds for nonlinear extremum problems and for MPECs with continuously differentiable data~\cite{GZL14}. Further extensions, such as the enhanced relaxed constant positive linear dependence condition (ERCPLD/rCPLD) and its MPEC variants (MPEC-RCPLD, MPEC-CRSC)~\cite{GLY13}, provide error-bound and calmness properties that are essential for the convergence analysis of various regularization and relaxation algorithms~\cite{ramos,GL13}.}

    On the other hand, there have been many recent publications on bilevel hyperparameter optimization (BHO). For instance, Bennett et al.  conducted a series of studies on hyperparameter selection based on transforming the problem into an MPEC \cite{Bennett2006,Kunapuli2008a,Kunapuli2008b,li2022}. 
    Li et al. \cite{ref1} formulated the support vector classification (SVC) hyperparameter selection problem as a bilevel optimization problem and then transformed it into an MPEC, which they solved with a relaxation algorithm that converges under the MPEC-MFCQ. 
    In \cite{li2023}, Qian et al. introduced an efficient linear-programming-based Newton global relaxation method to solve the MPEC derived from the bilevel model in \cite{Kunapuli2008a}, which incorporates multiple hyperparameters for feature selection. Furthermore, they showed that, under appropriate conditions, the MPEC-MFCQ defined in \cite{ref1} holds. In \cite{coniglio2023}, the authors further  characterized the MPEC-LICQ for the MPEC arising from the bilevel model used for hyperparameter selection in the context of the kernel-based SVC hyperparameter optimization problem. 
    Furthermore, in \cite{samuel}, Ward et al. improved solution methods for MPECs with new convergence results under weaker assumptions and formulated SVM hyperparameter tuning as an MPCC to leverage derivative information for efficiency. 
    Wang and Li \cite{yx24} formulated the hyperparameter selection for SVC with logistic loss as a bilevel optimization problem, transformed it into a single-level NLP via KKT conditions, and proposed a smoothing Newton method with superlinear convergence to efficiently obtain strict local minimizers.

    To summarize, with various CQs tailored to MPEC, the first aim of this paper is to take a close look at them and study the possible connections among them. In doing so, we identify those that are equivalent to each other (especially in the context of the multiple MFCQ--type CQs) and those that stand in strict implication relationships. Secondly, we study the behavior of these CQs for MPEC reformulation of (BHO) in support vector machines. 

    The contributions of the paper are as follows. First, for a general MPEC, we systematically analyze the relationships among several tailored constraint qualifications (such as MPEC-LICQ, MPEC-MFCQ-T, NNAMCQ and MPEC-MFCQ-R), establish their relationships and provide counterexamples showing where some reverse implications fail. Second, for the specific MPEC arising from the bilevel L1-loss support vector classification model, we identify conditions under which MPEC-LICQ, MPEC-MFCQ-T, MPEC-RCRCQ, MPEC-CPLD, MPEC-MFCQ-R and calmness property  hold or fail, thus offering practical ways for verifying these CQs in SVM-type bilevel formulations.
	
	The organization of the paper is as follows. In Section \ref{sec2},  we introduce some preliminaries for various definitions of constraint qualifications for MPECs. In Section \ref{sec3}, we study the relationships among them. In Section \ref{sec4}, after a brief introduction to the bilevel model for hyperparameter selection and its resulting MPEC in \cite{ref1}, we show that calmness property, MPEC-CPLD and MPEC-RCRCQ hold automatically for this MPEC, and that under certain conditions, MPEC-LICQ, MPEC-MFCQ-T and MPEC-MFCQ-R also hold.
    The final conclusions are presented in Section \ref{sec-6}.

    $\mathbf{Notations.}$
    % For $ x \in \mathbb{R}^{n}$, 
    We use $\|x\|_{0}$ to denote the number of nonzero elements in $ x \in \mathbb{R}^n$, while 
    $ \|x\|_{2} $ corresponds to the  $l_{2} $-norm of $ x $. 
    We denote  $x_{+}=\left(\left(x_{1}\right)_{+}, \cdots,\left(x_{n}\right)_{+}\right) \in \mathbb{R}^{n} $, where  $\left(x_{i}\right)_{+}=\max \left(x_{i}, 0\right) $.  
    $|\Omega| $ denotes the number of elements in the set $ \Omega \subset \mathbb{R}^{n} $. 
    For any positive integer $k$, we define $[k] := \{1,2,\ldots,k\}$.
   We denote by $\mathbb{B}_\varepsilon(v^*)
:= \{x \in \mathbb{R}^n \mid \|x - v^*\| \le \varepsilon\}$ the closed ball centered at \(v^*\) with radius \(\varepsilon\).
    We use $\mathbf{1}_{k} $ to denote a vector with elements all ones in $ \mathbb{R}^{k} $  
    and  $\mathbf{0}_{k} $ stands for a zero vector in  $\mathbb{R}^{k} $. $\mathbf{0}_{(\tau, \kappa)} $ will be used for a submatrix of the zero matrix, where  $\tau$  is the index set of the rows and $ \kappa$  is the index set of the columns. Similarly,  $I_{(\tau, \tau)}$  corresponds to a submatrix of an identity matrix indexed by both rows and columns in the set  $\tau $. Finally,  $\Theta_{(\tau, \cdot)} $ represents a submatrix of the matrix $ \Theta $, where $ \tau $ is the index set of the rows.

	 \section{Constraint Qualifications for MPECs}\label{sec2}

	In this section, we will briefly introduce the definitions of different CQs for MPECs as well as their properties. 
    The MPEC takes the following form:    
	\begin{equation}\label{mpec0}
		\begin{array}{cl}
			\min\limits_{v\in \mathbb{R}^n}& f(v) \\
			\hbox{s.t.}& g_i(v)\le 0, \ \forall \ i \in [m], \\
			& h_i(v)=0, \ \forall \ i \in [p], \\
			&  G_{i}(v) \geq 0 , \ H_{i}(v) \geq 0,\ G_{i}(v)H_{i}(v) =0, \forall \ i \in [l],
		\end{array}
		\tag{MPEC} 
	\end{equation} 
	where  $f: \mathbb{R}^{n} \rightarrow \mathbb{R}$, $g_i: \mathbb{R}^{n} \rightarrow \mathbb{R}$, $h_i: \mathbb{R}^{n} \rightarrow \mathbb{R}$,
    $G_i: \mathbb{R}^{n} \rightarrow \mathbb{R}$ and $H_i: \mathbb{R}^{n} \rightarrow \mathbb{R}$ are continuously differentiable.
	Let $v^*$ be a feasible point for \eqref{mpec0}. Define the following index sets:
	\begin{eqnarray*}
        I_g(v^*)&:=&\{i\in[m] \ | \  g_i(v^*)=0\}, \\
		I_G(v^*)&:=&\{i\in[l] \ | \  G_i(v^*)=0,  \ H_i(v^*)>0\},  \\ 
		I_H(v^*)&:=&\{i\in[l] \ | \  G_i(v^*)>0,  \ H_i(v^*)=0\},  \\ 
		I_{GH}(v^*)&:=&\{i\in[l] \ | \  G_i(v^*)=0,  \ H_i(v^*)=0\}.
    \end{eqnarray*}
    For simplicity, we usually drop the dependence on $v^*$ and use $I_g,\ I_G,\ I_H,\ I_{GH}$ instead. We start with MPEC-LICQ for \eqref{mpec0}, then move to some weaker constraint qualifications tailored for  \eqref{mpec0}.
    \begin{definition}\label{def-licq0}  \cite[page 920]{Scholtes2001} 
		Let $v^*$ be a feasible point of \eqref{mpec0}. 
        MPEC-LICQ is said to hold at $v^*$  if the set of gradient vectors 
        \begin{equation} \label{eq22}
            \begin{aligned}
            &\left\{\nabla g_i(v^*)\ |\  i\in I_g\right\} \\
            \cup& \big\{\left\{\nabla h_i(v^*)\ |\ i\in [p]\right\} \cup\left\{\nabla G_{i}\left(v^*\right) \mid i \in  I_{G} \cup I_{GH} \right\}
            \cup \left\{\nabla H_{i}\left(v^*\right) \mid i \in  I_{H} \cup I_{GH}\right\}\big\}
            \end{aligned}
        \end{equation}

		is  linearly independent.
	\end{definition} 	

    One type of MPEC-MFCQ (we refer to it as MPEC-MFCQ-T) is defined based on the following \emph{tightened} nonlinear programming problem at a feasible point $v^*$:
	\begin{equation} \label{tnlp}
		\begin{array}{cl}
			\min\limits_{v\in \mathbb{R}^n}& f(v) \\
			\hbox{s.t.}& g_i(v)\le 0, \ \forall \  i \in [m], \\
			& h_i(v)=0, \ \forall \ i \in [p], \\
			&  G_{i}(v) = 0 , \  H_i(v)\ge0,  \  \forall \ i\in I_G, \\
			&  G_{i}(v) \ge 0,  \  H_i(v)=0,  \ \forall \ i\in I_H, \\
			&    G_{i}(v)=0,  \ H_{i}(v) =0, \ \forall \ i\in I_{GH}. 
		\end{array}
		\tag{TNLP($v^*$)}
	\end{equation}
	\begin{definition}\cite[page 263]{Hoheisel2013}\label{def-mfcq0} 
		Let $v^*$ be a feasible point of \eqref{mpec0}. 
        MPEC-MFCQ-T is said to hold at $v^*$ if the standard MFCQ holds for the corresponding \eqref{tnlp}. Equivalently, MPEC-MFCQ-T holds at $v^*$
        if the set of gradient vectors in \eqref{eq22} is positively linearly independent. That is, { there is no nonzero family of scalars }
        $\{\alpha_i\}_{i\in I_g}$, $\{\beta_i\}_{i\in[p]}$, $\{\gamma_i\}_{i\in I_G\cup I_{GH}}$, and $\{\eta_i\}_{i\in I_H\cup I_{GH}}$
        with
        $
        \alpha_i\ge 0,\  \forall\, i\in I_g,$
        %not all of them being zero, 
        such that
        $$
        \sum_{i\in I_g}\alpha_i \nabla g_i(v^*)
        +\sum_{i\in[p]}\beta_i \nabla h_i(v^*)
        +\sum_{i\in I_G\cup I_{GH}}\gamma_i \nabla G_i(v^*)
        +\sum_{i\in I_H\cup I_{GH}}\eta_i \nabla H_i(v^*)
        =0.$$

	\end{definition}

    In comparison, another type of MPEC-MFCQ \cite{Ralph04} (we refer to it as MPEC-MFCQ-R) is based on the following \emph{relaxed} nonlinear programming problem at the feasible point $v^*$:
	\begin{equation} \label{rnlp}
		\begin{array}{lll}
			\min\limits_{v\in \mathbb{R}^n}& f(v)&\\
			\hbox{s.t.} & g_i(v)\le 0,  \ &\forall \ i \in [m], \\
			& h_i(v)=0,  &\forall \ i \in [p], \\
			& 
			G_{i}(v) = 0,    \ &\forall \ i\in I_G, \\
			&H_{i}(v) = 0 ,  &  \forall \ i\in I_H, \\
			& G_{i}(v)\ge 0,  \ H_{i}(v) \ge 0, & \forall \ i\in I_{GH}.
		\end{array}
		\tag{RNLP($v^*$)}
	\end{equation}

    {
    % \color{blue}
    \begin{definition}[MPEC-MFCQ-R]\label{def-mpec-mfcq-r}
    Let $v^*$ be a feasible point of \eqref{mpec0}. 
        MPEC-MFCQ-R is said to hold at $v^*$ if the standard MFCQ is satisfied for  \eqref{rnlp}; that is,  the following  set of gradient vectors 
        \begin{equation}\label{grad-mfcqr}
        \begin{aligned}
        &\Bigl\{\{\nabla g_i(v^*)\mid i\in I_g\}
         \cup
       \{-\nabla G_i(v^*)\mid i\in I_{GH}\}
         \cup
        \{-\nabla H_i(v^*)\mid i\in I_{GH}\}\Bigr\} \\
        \cup\;&
        \Bigl\{\{\nabla h_i(v^*)\mid i\in[p]\}
        \cup
        \{\nabla G_i(v^*)\mid i\in I_G\}
        \cup
       \{\nabla H_i(v^*)\mid i\in I_H\}\Bigr\}.
        \end{aligned}
        \end{equation}
        is positively linearly independent. That is, there do not exist scalars
        $\{\alpha_i\}_{i\in I_g}$,
        $\{\beta_i\}_{i\in[p]}$,
        $\{\gamma_i\}_{i\in I_G}$,
        $\{\eta_i\}_{i\in I_H}$,
        $\{\xi_i\}_{i\in I_{GH}}$,
        and $\{\zeta_i\}_{i\in I_{GH}}$
        with
        $\alpha_i\ge 0,\ \forall\, i\in I_g$, $
        \xi_i\ge 0,\ \forall\, i\in I_{GH}$, $
        \zeta_i\ge 0,\ \forall\, i\in I_{GH}$,
        not all of them being zero, such that
        \begin{equation}\label{eq_mfcqr_pli}
        \begin{aligned}
        &\sum_{i\in I_g}\alpha_i \nabla g_i(v^*)
        -\sum_{i\in I_{GH}}\xi_i \nabla G_i(v^*)
        -\sum_{i\in I_{GH}}\zeta_i \nabla H_i(v^*) \\
        +&\sum_{i\in[p]}\beta_i \nabla h_i(v^*)
        +\sum_{i\in I_G}\gamma_i \nabla G_i(v^*)
        +\sum_{i\in I_H}\eta_i \nabla H_i(v^*)
        =0.
        \end{aligned}
        \end{equation}
        \end{definition}
        }
	Obviously, \eqref{tnlp} and \eqref{rnlp}  depend on the chosen point $v^*$. Note that \eqref{tnlp} is referred to as \emph{tightened} because its feasible region is a subset of the feasible region of \eqref{mpec0}. 
    %However, there is no inclusion relationship between the feasible regions of \eqref{rnlp} and \eqref{mpec0}.
    { The feasible region of  \eqref{mpec0} is locally contained within the feasible region of  \eqref{rnlp}.}
    Moreover, if  $v^*$  is a local minimizer of \eqref{rnlp}, then it is also a local minimizer of \eqref{mpec0}. If  $v^*$  is a local minimizer of \eqref{mpec0}, then it is a local minimizer of \eqref{tnlp}\cite{Scheel2000}. However, the reverse implications hold only under a strict complementarity condition at $v^*$.
	
	Classical MPEC constraint qualifications (e.g., MPEC-MFCQ-T and MPEC-LICQ) are often too restrictive to be satisfied in practical MPECs. To address this, generalized CQs have been introduced that relax certain assumptions and broaden the applicability. A flexible variant was proposed by Ye \cite{JaneYe05}, 
     termed the MPEC generalized MFCQ (MPEC-GMFCQ), 
    which is given below.

	\begin{definition}\label{def-gmfcq}\cite[Definition 2.11]{JaneYe05} 
		Let $v^*$ be a feasible point of \eqref{mpec0}. 
        MPEC-GMFCQ is said to hold at $v^*$ if the following conditions hold:
		\begin{itemize}
			\item [$\left(\mathrm{i}\right)$] For every partition of $I_{GH}$ into sets $P, Q, R$ with $R\neq \emptyset$,  there exists a vector $d\in\mathbb R^{n }$ such that 
			\begin{equation*}
				\left\{\begin{array}{ll}
					\nabla g_i(v^*)^\top d\le 0,  \ &\forall \ i\in I_g, \\
					\nabla h_i(v^*)^\top d= 0,  \ &\forall \ i\in [p], \\\nabla G_i(v^*)^\top d=0,  & \forall \ i\in I_G\cup Q, \\
					\nabla H_i(v^*)^\top d=0,  & \forall \ i\in I_H\cup P, \\
					\nabla
					G_i(v^*)^\top d\ge0, \ \nabla
					H_i(v^*)^\top d\ge0,  & \forall \ i\in R, \end{array}\right.
			\end{equation*}
			and for some $i\in R$ either $\nabla G_i(v^*)^\top d>0 $ or $ \nabla
			H_i(v^*)^\top d>0$.
			\item [$\left(\mathrm{ii}\right)$] For every partition of $I_{GH}$ into sets $P$ and $Q$,  the set of gradient vectors 
			\begin{equation*}\label{grad-gmfcq}
                    \{\nabla h_i(v^*) \mid i \in [p]\}
                    \cup
                    \{\nabla G_i(v^*) \mid i \in I_G \cup Q\}
                    \cup
                    \{\nabla H_i(v^*) \mid i \in I_H \cup P\}
            \end{equation*}
			is linearly independent and there exists $d\in\mathbb R^{n}$ such that 
			\begin{equation*}\label{cond-ii-2}
				\left\{\begin{array}{ll}
					\nabla g_i(v^*)^\top d< 0,  \ &\forall \ i\in I_g, \\
					\nabla h_i(v^*)^\top d= 0,  \ &\forall \ i\in [p], \\
					\nabla G_i(v^*)^\top d=0,  & \forall \ i\in I_G\cup Q, \\
					\nabla H_i(v^*)^\top d=0,  & \forall \ i\in I_H\cup P.\\
				\end{array}\right.
			\end{equation*}
		\end{itemize}
	\end{definition}
	
	Another useful CQ is the No Nonzero Abnormal Multiplier Constraint Qualification (NNAMCQ) which is equivalent to MPEC-GMFCQ \cite{JaneYe05}.
	\begin{definition}\label{def-nnamcq}\cite[Definition 2.10]{JaneYe05}
		Let $v^*$ be a feasible point of \eqref{mpec0}. 
        NNAMCQ is said to hold at $v^*$ if there is no nonzero vector $(\lambda^g, \lambda^h, \lambda^G, \lambda^H)\in\mathbb R^{m+p+2l}$ such that 
		\begin{equation}\label{eq-nnamcq}
			\left\{
			\begin{aligned}
				&0=\sum_{i\in I_g}\lambda_i^g\nabla g_i(v^*) + \sum_{i\in [p]}\lambda_i^h\nabla h_i(v^*) -\sum_{i=1}^l(\lambda_i^G\nabla G_i(v^*) +\lambda_i^H\nabla H_i(v^*) ),  \\
				& \lambda^g_{I_g}\ge0,  \ \lambda_{I_H}^G=0,  \ \lambda_{I_G}^H=0,  \\
				&\hbox{either }\lambda_i^G>0,  \ \ \lambda_i^H>0,  \hbox{ or }\lambda_i^G\lambda_i^H=0, \   \forall \ i\in I_{GH}.
			\end{aligned}
			\right.
		\end{equation}
	\end{definition}

    Other relatively relaxed constraint qualifications include: MPEC Abadie constraint qualification (MPEC-ACQ),  MPEC constant positive linear dependence condition (MPEC-CPLD) and  MPEC relaxed constant rank constraint qualification (MPEC-RCRCQ), which are given below.

    \begin{definition}\label{def-acq0}\cite[Definition 3.1]{MLF1}
Let $v^*$ be a feasible point of \eqref{mpec0}, and let $\mathcal V$ denote the feasible region of \eqref{mpec0}. 
MPEC-ACQ is said to hold at $v^*$ if 
\[
  \mathcal T_{\operatorname{MPEC}}^{\operatorname{lin}}(v^*) = \mathcal T(v^*),
\]
where the tangent cone of \eqref{mpec0} at $v^*$ is given by
\[
  \mathcal T(v^*)
  := \bigl\{ d \in \mathbb{R}^n \,\big|\,
  \exists\,\{v^k\} \subset \mathcal V,\ \exists\, t_k \downarrow 0
  \text{ such that } v^k \to v^*,\ (v^k - v^*)/t_k \to d \bigr\},
\]
and $\mathcal T_{\operatorname{MPEC}}^{\operatorname{lin}}(v^*)$ is the MPEC-linearized tangent cone of \eqref{mpec0} at $v^*$, defined by
\[
\mathcal T_{\operatorname{MPEC}}^{\operatorname{lin}}(v^*):=\left\{d\in\mathbb R^{n}:
\begin{array}{cc}
\nabla g_i(v^*)^\top d\le0,  & i\in I_g, \\
\nabla h_i(v^*)^\top d = 0, & i\in [p], \\
\nabla G_i(v^*)^\top d=0,  & i\in I_G, \\
\nabla H_i(v^*)^\top d = 0, & i\in I_H, \\
\nabla G_i(v^*)^\top d\ge0,  & i\in I_{GH}, \\
\nabla H_i(v^*)^\top d \ge 0, & i\in I_{GH}, \\
\bigl(\nabla G_i(v^*)^\top d\bigr)\bigl(\nabla H_i(v^*)^\top d\bigr) = 0, & i\in I_{GH}
\end{array}\right\}.
\]
\end{definition}

    { 
    
    \begin{definition}\cite[Definition 3.2]{TCA12}\label{def_MPECCPLD} 
        Let $v^*$ be a feasible point of \eqref{mpec0}.  MPEC-CPLD is said to
        hold at $v^*$ if, for every choice of index subsets
        $I_{1} \subseteq I_{g}$, 
        $I_{2} \subseteq [p]$, 
        $I_{3} \subseteq I_{GH} \cup I_{G}$, 
        $I_{4} \subseteq I_{GH} \cup I_{H}$,
        the following implication holds: If the set of gradient vectors
        \begin{equation}\label{def-mpec-cpld}
            \left\{\nabla g_{i}(v^{*}) \mid i \in I_{1}\right\}
        \cup
        \big\{\left\{\nabla h_{i}(v^{*}) \mid i \in I_{2}\right\}
        \cup
        \left\{\nabla G_{i}(v^{*}) \mid i \in I_{3}\right\}
        \cup
        \left\{\nabla H_{i}(v^{*}) \mid i \in I_{4}\right\}\big\}
        \end{equation}
        is positively linearly dependent at $v^{*}$, then there exists a neighborhood
        $\mathcal N(v^{*})$ such that  \eqref{def-mpec-cpld}
        remains linearly dependent for all $v \in \mathcal N(v^{*})$.
        \end{definition}

    \begin{definition}\cite[Definition 3.4]{GLY13} 
        Let $v^*$ be a feasible point of \eqref{mpec0}.  
        MPEC-RCRCQ is said to
        hold at $v^*$ if there exists $\delta > 0$ such that, 
        for each index subsets 
        $I_{1} \subseteq I_{g}$, 
        $I_{2} \subseteq [p]$, 
        $I_{3} \subseteq I_{GH} \cup I_{G}$, and 
        $I_{4} \subseteq I_{GH} \cup I_{H}$ 
        with the set of gradient vectors
        \begin{equation*}
        \left\{\nabla g_{i}\left(v\right) \mid i \in I_{1}\right\}
        \cup\left\{\nabla h_{i}\left(v\right) \mid i \in I_{2}\right\}
        \cup\left\{\nabla G_{i}\left(v\right) \mid i \in I_{3}\right\}
        \cup\left\{\nabla H_{i}\left(v\right) \mid i \in I_{4}\right\}
        \end{equation*}
        has the same rank for all $v \in \mathbb{B}_\delta(v^*)$.
        \end{definition}
        
}

{ 
    Besides the CQs introduced above, we also refer to the MPEC variants relaxed constant positive linear dependence constraint (MPEC-rCPLD and MPEC-RCPLD), MPEC constant rank constraint qualification (MPEC-CRCQ), relaxed no nonzero abnormal multiplier constraint qualification (RNNAMCQ), the MPEC constant positive linear dependence condition (MPEC-CCP), and the MPEC constant rank of the subspace component condition (MPEC-CRSC). Compared with MPEC-CPLD, MPEC-rCPLD restricts the positive linear dependence check to gradient families generated by a fixed basis of the equality constraints, whereas MPEC-RCPLD combines a relaxed constant rank requirement on selected gradients with an additional local implication involving multipliers. Compared with MPEC-RCRCQ, MPEC-CRCQ requires the constant rank requirement for all active index subsets. Compared with NNAMCQ, the relaxed NNAMCQ (RNNAMCQ) only rules out abnormal multipliers that satisfy an additional constant-rank structure in a neighbourhood of $v^{*}$. The condition MPEC-CCP is formulated as an outer semicontinuity requirement at $(v^{*},F(v^{*}))$ for the set-valued mapping $(x,z)\rightrightarrows \nabla F(x)^{\top}N_{\Pi}(z)$, while MPEC-CRSC combines a local constant-rank property of selected active constraint gradients with an additional multiplier-based implication that enforces linear dependence of corresponding gradient subfamilies near $v^{*}$. Since we only use these notions through standard implication chains, we do not reproduce their definitions here. Precise definitions can be found in \cite[Definition 3.1]{CL13}, \cite[Definitions 4.3 and 3.4]{GLY13}, \cite[Definition 3.9]{ramos}, and \cite[Definition 3.1 (1) and (2)]{Alberto19}.
    }

    We next present another equivalent formulation of \eqref{mpec0} and introduce the constraint  regularity properties associated with this formulation. 
     
     Define the function $F(v):=\big(g(v),\ h(v),\ \big(G(v),\ H(v)\big)\big)$ and the set $\Pi:=\mathbb{R}_{-}^{m}\times \{0\}^{p}\times \mathcal{C}^{l}$ with 
    $\mathcal{C}:=\left\{(a, b) \in \mathbb{R}^{2} \mid a \geq 0,\ b \geq 0,\ ab=0\right\}$, where 
    \(\mathcal{C}^{l}\) denotes the l-fold Cartesian product of the set \(\mathcal{C}\). 
     \eqref{mpec0} is equivalent to  the following problem
	\begin{equation}\label{ca}
		\min f(v) \quad \text { s.t. } \quad F(v) \in \Pi.    
	\end{equation}
    
   Define the multifunction $\Phi: \mathbb{R}^{m+p+2l} \rightrightarrows \mathbb{R}^{n}$ associated with the constraint system of \eqref{ca} by
    \begin{equation}\label{mpec_cons_map}
        \Phi(u):=\left\{z \in \mathbb{R}^{n} \mid F(z)+u \in \Pi\right\} .
    \end{equation}
    Beyond traditional CQs, the calmness property provides a stability characterization for the constraint system.

	\begin{definition}\cite[Definition 2]{MFL5} 
		Let $\Phi: \mathbb{R}^{m+p+2l} \rightrightarrows \mathbb{R}^{n}$ be a multifunction with closed graph and let $(u,v)\in \operatorname{gph}\Phi$, where 
        $$\operatorname{gph} \Phi:=\left\{(u,v)\in \mathbb{R}^{m+p+2l}\times \mathbb{R}^n \mid v\in \Phi(u)\right\}.$$
        $\Phi$ is calm at  $(u, v)$  provided there exist neighborhoods  $\mathcal{U}$  of  $u$, $\mathcal{V}$  of  $v $, and a modulus  $L \geqslant 0$  such that
		$$
		\Phi\left(u^{\prime}\right) \cap \mathcal{V} \subset \Phi(u)+L\|u^{\prime}-u\| \mathbb{B}_1(0), \quad \forall \  u^{\prime} \in \mathcal{U}.
		$$
	\end{definition}
    { 
    The Aubin property is strictly stronger than calmness and requires a Lipschitz estimate relating $\Phi(u_1)$ and $\Phi(u_2)$ for all $u_1,u_2$ in a neighbourhood. We do not elaborate on its definition here and refer to \cite[Definition~9.36]{Rockafellar98} for the precise formulation.
    }
    { 
    \begin{remark}\label{affinecalmness}
    When all constraint functions $g_i,h_i,G_i,H_i$ are affine of $v$, each complementarity pair splits the cone $\mathcal{C}$ into two polyhedral branches, equivalently $\mathcal{C}=\{(a,b)\in\mathbb{R}^2 \mid a=0,\ b\ge 0\}\ \cup\ \{(a,b)\in\mathbb{R}^2 \mid a\ge 0,\ b=0\}.$ Hence $\Pi$ is a finite union of polyhedral sets and, since $F$ is affine, the  $\operatorname{gph}\Phi$ is also a finite union of polyhedral sets. It is well known that multifunctions whose graphs are finite unions of polyhedral sets are upper–Lipschitz (and thus calm) at every point of their graphs; see, for example, \cite[Proposition~1]{Robinson09} and the discussion in \cite{Uderzo21}. Consequently, $\Phi$ is calm at every $(\bar u,\bar z)\in\operatorname{gph}\Phi$ in the affine case mentioned here.
    \end{remark}
    \begin{remark}
    Moreover, the calmness of $\Phi$ at $(0,\bar z)$ is equivalent to the existence of a local error bound for $F$ relative to $\Pi$, i.e., there exist $\kappa,\varepsilon>0$ such that $\operatorname{dist}\bigl(z,\Phi(0)\bigr)\;\le\;\kappa\,\operatorname{dist}\bigl(F(z),\Pi\bigr),\ \forall\, z\in \mathbb{B}_\varepsilon(\bar z)$  \cite{Henrion}.  
    \end{remark}
}

	    { \section{Relationships for MPEC Constraint Qualifications} \label{sec3}}
	
	Having introduced various constraint qualifications, we now investigate their connections.
    
         {\subsection{Connections Between MPEC Constraint Qualifications}}
    There are two immediate implication results.
    \begin{proposition}\label{lt}\textnormal{\cite{JaneYe05}}
		At a feasible point $v^*$ of \eqref{mpec0}, the following results hold:
        \begin{itemize}
            \item[$\left(\mathrm{i}\right)$] If MPEC-LICQ holds at $v^*$,  then MPEC-MFCQ-T holds at $v^*$.
            \item[$\left(\mathrm{ii}\right)$] If MPEC-MFCQ-T holds at $v^*$,  then MPEC-GMFCQ holds at $v^*$.
        \end{itemize}   
	\end{proposition}
	
    Next, comparing the \emph{tightened} and \emph{relaxed} formulations (Definitions \ref{def-mfcq0} and \ref{def-mpec-mfcq-r}) yields the following result.

    \begin{proposition}
		At a feasible point $v^*$ of \eqref{mpec0}, the following results hold:
        \begin{itemize}
			\item[$\left(\mathrm{i}\right)$] If  MPEC-MFCQ-T holds at $v^*$,  then MPEC-MFCQ-R holds at $v^*$.
			\item[$\left(\mathrm{ii}\right)$] If strict complementarity condition holds at $v^*$,   MPEC-MFCQ-T is equivalent to  MPEC-MFCQ-R.
		\end{itemize}
	\end{proposition}

	\begin{proof}
		$\left(\mathrm{i}\right)$ is { trivial}. For $\left(\mathrm{ii}\right)$, note that if strict  complementarity condition holds at $v^*$,   $I_{GH}=\emptyset$. In that case,   MPEC-MFCQ-R holds if 
		the following set of vectors
		\begin{equation*}
			\begin{aligned}
				&\big\{\nabla g_i(v^*)\ |\ i\in I_g\big\} \cup\big\{\left\{\nabla h_i(v^*)\ |\ i\in [p]\right\}\cup \left\{\nabla G_i(v^*)\ |\ i\in I_{G}\right\}\cup\left\{\nabla H_i(v^*)\ |\ i\in I_{H}\right\}\big\}
			\end{aligned}
		\end{equation*}
		is positively linearly independent. MPEC-MFCQ-T holds if the above set of vectors is positively linearly independent as well. Therefore,  in this case,  MPEC-MFCQ-T is equivalent to  MPEC-MFCQ-R.
	\end{proof}

	\begin{proposition}\label{nnamcqeqgmfcq} \textnormal{\cite[Proposition 2.1]{JaneYe05}} NNAMCQ is equivalent to  MPEC-GMFCQ.
	\end{proposition}
    We next establish the relationship among MPEC-GMFCQ and MPEC-MFCQ-R.
	
	\begin{proposition}\label{gr}
		At a feasible point $v^*$ of \eqref{mpec0}, the following results hold:
		\begin{itemize}
			\item[$\left(\mathrm{i}\right)$] If  MPEC-GMFCQ holds at $v^*$,  then  MPEC-MFCQ-R holds at $v^*$.
			\item[$\left(\mathrm{ii}\right)$] In particular,  if strict complementarity condition holds at $v^*$,  MPEC-GMFCQ is equivalent to  MPEC-MFCQ-R.
		\end{itemize}
	\end{proposition}
	
	\begin{proof}
		{  
            \noindent $\left(\mathrm{i}\right)$ By Proposition \ref{nnamcqeqgmfcq}, MPEC-GMFCQ is equivalent to NNAMCQ.
    We prove $\left(\mathrm{i}\right)$ by showing that NNAMCQ implies MPEC-MFCQ-R.
    Assume that MPEC-MFCQ-R fails at $v^*$. By Definition \ref{def-mpec-mfcq-r},
    there exist scalars
    $\{\alpha_i\}_{i\in I_g}$,
    $\{\beta_i\}_{i\in[p]}$,
    $\{\gamma_i\}_{i\in I_G}$,
    $\{\eta_i\}_{i\in I_H}$,
    $\{\xi_i\}_{i\in I_{GH}}$,
    $\{\zeta_i\}_{i\in I_{GH}}$,
    not all zero, such that
    $\alpha_i\ge 0$ for all $i\in I_g$,
    $\xi_i\ge 0$ and $\zeta_i\ge 0$ for all $i\in I_{GH}$, and \eqref{eq_mfcqr_pli} holds. 
    Now define $(\lambda^g,\lambda^h,\lambda^G,\lambda^H)\in\mathbb{R}^{m+p+2l}$ by
    \[
    \lambda_i^g:=\alpha_i\ (i\in I_g),\ 
    \lambda_i^h:=\beta_i\ (i\in[p]),\ 
    \lambda_i^G :=
    \begin{cases}
    -\gamma_i, & i\in I_G,\\
    0, & i\in I_H,\\
    \xi_i, & i\in I_{GH},
    \end{cases}
    \ 
    \lambda_i^H :=
    \begin{cases}
    0, & i\in I_G,\\
    -\eta_i, & i\in I_H,\\
    \zeta_i, & i\in I_{GH}.
    \end{cases}
    \]
    Then $\lambda^g_{I_g}\ge 0$, $\lambda^G_{I_H}=0$, and $\lambda^H_{I_G}=0$.
    Moreover, for each $i\in I_{GH}$ we have $\lambda_i^G=\xi_i\ge 0$ and
    $\lambda_i^H=\zeta_i\ge 0$, hence either $\lambda_i^G\lambda_i^H=0$ or
    $\lambda_i^G>0$ and $\lambda_i^H>0$.
    Finally, substituting the above definitions into \eqref{eq_mfcqr_pli} yields
    \[
    0=\sum_{i\in I_g}\lambda_i^g\nabla g_i(v^*)
      +\sum_{i\in[p]}\lambda_i^h\nabla h_i(v^*)
      -\sum_{i=1}^l\bigl(\lambda_i^G\nabla G_i(v^*)+\lambda_i^H\nabla H_i(v^*)\bigr),
    \]
    so $(\lambda^g,\lambda^h,\lambda^G,\lambda^H)\neq 0$ satisfies \eqref{eq-nnamcq}.
     NNAMCQ fails at $v^*$.
    Therefore, if NNAMCQ holds at $v^*$, then MPEC-MFCQ-R  holds at $v^*$.
    }
    
        \vspace{2.5mm}
	\noindent $\left(\mathrm{ii}\right)$ If $I_{GH}=\emptyset$,  the definition of MPEC-GMFCQ reduces to the following statement.  The set of gradient vectors
		\begin{equation*}
			\{\nabla h_i(v^*) \mid i \in [p]\}
                    \cup
                    \{\nabla G_i(v^*) \mid i \in I_G \}
                    \cup
                    \{\nabla H_i(v^*) \mid i \in I_H \}
		\end{equation*}
		is linearly independent and there exists $d\in\mathbb R^{n}$ such that 
		\begin{equation*}
			\left\{\begin{array}{ll}
				\nabla g_i(v^*)^\top d< 0,  \ &\forall \ i\in I_g, \\
				\nabla h_i(v^*)^\top d= 0,  \ &\forall \ i\in [p], \\
				\nabla G_i(v^*)^\top d=0,  & \forall \ i\in I_G, \\
				\nabla H_i(v^*)^\top d=0,  & \forall \ i\in I_H, \\
			\end{array}\right.
		\end{equation*}
		which is exactly standard MFCQ for the \emph{relaxed} problem \eqref{rnlp} when $I_{GH}=\emptyset$. Therefore,  if $I_{GH}=\emptyset$,    MPEC-GMFCQ is equivalent to  MPEC-MFCQ-R.           
	\end{proof}

    { \subsection{Some Counterexamples}}
	Next, we provide illustrative counterexamples that clarify how the main constraint qualifications for MPECs relate to each other. 

\begin{example}{(MPEC-MFCQ-T $\centernot\Longrightarrow$ MPEC-LICQ)}
Consider the MPEC
$$
\begin{array}{ll}
\displaystyle \min_{v \in \mathbb{R}^3} & f(v)=v_2+v_3 \\[1mm]
\mathrm{s.t.} & g_1(v)=v_1\le 0, \\
& g_2(v)=v_1+v_2 \le 0, \\
& G(v)=v_2 \ge 0, \\
& H(v)=v_3 \ge 0, \\
& G(v)\cdot H(v)=v_2 v_3=0 .
\end{array}
$$
At $v^*=(0,0,0)^\top$,  $I_g=\{1,2\}$, $I_G=I_H=\emptyset$, and $I_{GH}=\{1\}$.  
The corresponding \eqref{tnlp} is
$$
\begin{array}{cl}
\displaystyle \min_{v \in \mathbb{R}^3}  & v_2+v_3 \\[1mm]
\mathrm{s.t.} & g_1(v)=v_1\le 0, \\
& g_2(v)=v_1+v_2 \le 0, \\
& G(v)=v_2 = 0, \\
& H(v)=v_3 = 0 .
\end{array}
$$
The gradients of the active constraints are
$\nabla g_1(v^*)^\top=(1,0,0)$, $\nabla g_2(v^*)^\top=(1,1,0)$, 
$\nabla G(v^*)^\top=(0,1,0)$, and $\nabla H(v^*)^\top=(0,0,1)$.  
The set $\{\nabla g_1(v^*),\nabla g_2(v^*)\}\cup\{\nabla G(v^*),\nabla H(v^*)\}$ is
positively linearly independent but linearly dependent.  
Therefore, MFCQ holds for \eqref{tnlp}, while LICQ fails.  
Consequently,  the MPEC-MFCQ-T holds at $v^*$, but  the MPEC-LICQ fails.
\end{example}

\begin{example}{(NNAMCQ $\centernot\Longrightarrow$ MPEC-MFCQ-T)}
 Consider the MPEC
$$
\begin{array}{ll}
\displaystyle \min_{v \in \mathbb{R}^2} & f(v)=v_1+v_2 \\[1mm]
\mathrm{s.t.} & g_1(v)=-v_1-v_2 \le 0, \\
& G(v)=v_1 \ge 0, \quad H(v)=v_2 \ge 0, \quad G(v)H(v)=v_1v_2=0 .
\end{array}
$$
At $v^*=(0,0)^\top$,  $I_g=\{1\}$, $I_G=\emptyset$, $I_H=\emptyset$, and $I_{GH}=\{1\}$.  
The gradients are $\nabla g_1(v^*)^\top=(-1,-1)$, $\nabla G(v^*)^\top=(1,0)$, and $\nabla H(v^*)^\top=(0,1)$.  
By Definition~\ref{def-nnamcq}, there is no nonzero vector $(\lambda^g,\lambda^G,\lambda^H)\in\mathbb R\times\mathbb R\times\mathbb R$ such that
\[
\left\{
\begin{aligned}
&0=\lambda^g \!\!\begin{bmatrix}-1\\ -1\end{bmatrix}
 -\lambda^G\!\!\begin{bmatrix}1\\ 0\end{bmatrix}
 -\lambda^H\!\!\begin{bmatrix}0\\ 1\end{bmatrix}
 =\begin{bmatrix}-\lambda^g-\lambda^G \\ -\lambda^g-\lambda^H\end{bmatrix},\\
& \lambda^g\ge 0,\\
& \text{either }\lambda^G>0,\ \lambda^H>0,\ \text{or }\lambda^G\lambda^H=0 .
\end{aligned}
\right.
\]
Therefore, NNAMCQ holds at $v^*$.  
However, the set $\{\nabla g_1(v^*)\}\cup\{\nabla G(v^*),\nabla H(v^*)\}$ is positively linearly dependent, so  the MPEC-MFCQ-T fails at $v^*$.
\end{example}

\begin{example}{ (MPEC-MFCQ-R $\centernot\Longrightarrow$ NNAMCQ)}
Consider the MPEC
$$
\begin{array}{ll}
\displaystyle \min_{v \in \mathbb{R}^3} & f(v)=v_1+v_2+v_3 \\[1mm]
\mathrm{s.t.} & G_1(v)=v_1 \ge 0, \quad  H_1(v)=v_2 \ge 0, \quad G_1(v)H_1(v)=0, \\
& G_2(v)=v_1-v_2^2 \ge 0, \quad  H_2(v)=v_3 \ge 0, \quad G_2(v)H_2(v)=0. \\
\end{array}
$$
At $v^*=(0,0,0)^\top$,  $I_{GH}=\{1,2\}$ and $I_G=I_H=\emptyset$.  
The corresponding \eqref{rnlp} is
$$
\begin{array}{ll}
\displaystyle \min_{v \in \mathbb{R}^3} & f(v)=v_1+v_2+v_3 \\[1mm]
\mathrm{s.t.} & G_1(v)=v_1 \ge 0, \quad H_1(v)=v_2 \ge 0, \\
& G_2(v)=v_1-v_2^2 \ge 0, \quad H_2(v)=v_3 \ge 0 .
\end{array}
$$
The active gradients are
$\nabla G_1(v^*)^\top=(1,0,0)$, $\nabla H_1(v^*)^\top=(0,1,0)$,
$\nabla G_2(v^*)^\top=(1,0,0)$,  $\nabla H_2(v^*)^\top=(0,0,1)$.  
The set $\{-\nabla G_1(v^*),-\nabla H_1(v^*),-\nabla G_2(v^*),-\nabla H_2(v^*)\}$ is positively linearly independent. The MPEC-MFCQ-R holds at $v^*$.  
However, there exists a nonzero vector $(\lambda_1^{G},\lambda_1^{H},\lambda_2^{G},\lambda_2^{H})=(1,0,-1,0)^\top$ satisfying
\[
\left\{
\begin{aligned}
&0=-\lambda_1^G\!\!\begin{bmatrix}1\\0\\0\end{bmatrix}
   -\lambda_1^H\!\!\begin{bmatrix}0\\1\\0\end{bmatrix}
   -\lambda_2^G\!\!\begin{bmatrix}1\\0\\0\end{bmatrix}
   -\lambda_2^H\!\!\begin{bmatrix}0\\0\\1\end{bmatrix}
   =\begin{bmatrix}-\lambda_1^G-\lambda_2^G\\ -\lambda_1^H\\ -\lambda_2^H\end{bmatrix},\\
& \text{either }\lambda_i^G>0,\ \lambda_i^H>0,\ \text{or }\lambda_i^G\lambda_i^H=0,\ \ i\in I_{GH},
\end{aligned}
\right.
\]
hence NNAMCQ fails at $v^*$.
\end{example}

	Based on these properties and counterexamples, together with additional implication relations available in \cite{CL13,GL13,GLY13,GZL14,Alberto19,ramos,TCA12,MFL5,Henrion18,Kanzow10} and further counterexamples reported in \cite{Henrion,CL13,Alberto19},  we can summarize the implications among the CQs at a feasible point $v^*$ of \eqref{mpec0} in Figure \ref{cqsrelation}.

    \begin{figure}[H]
        \centering
        \includegraphics[width=\linewidth]{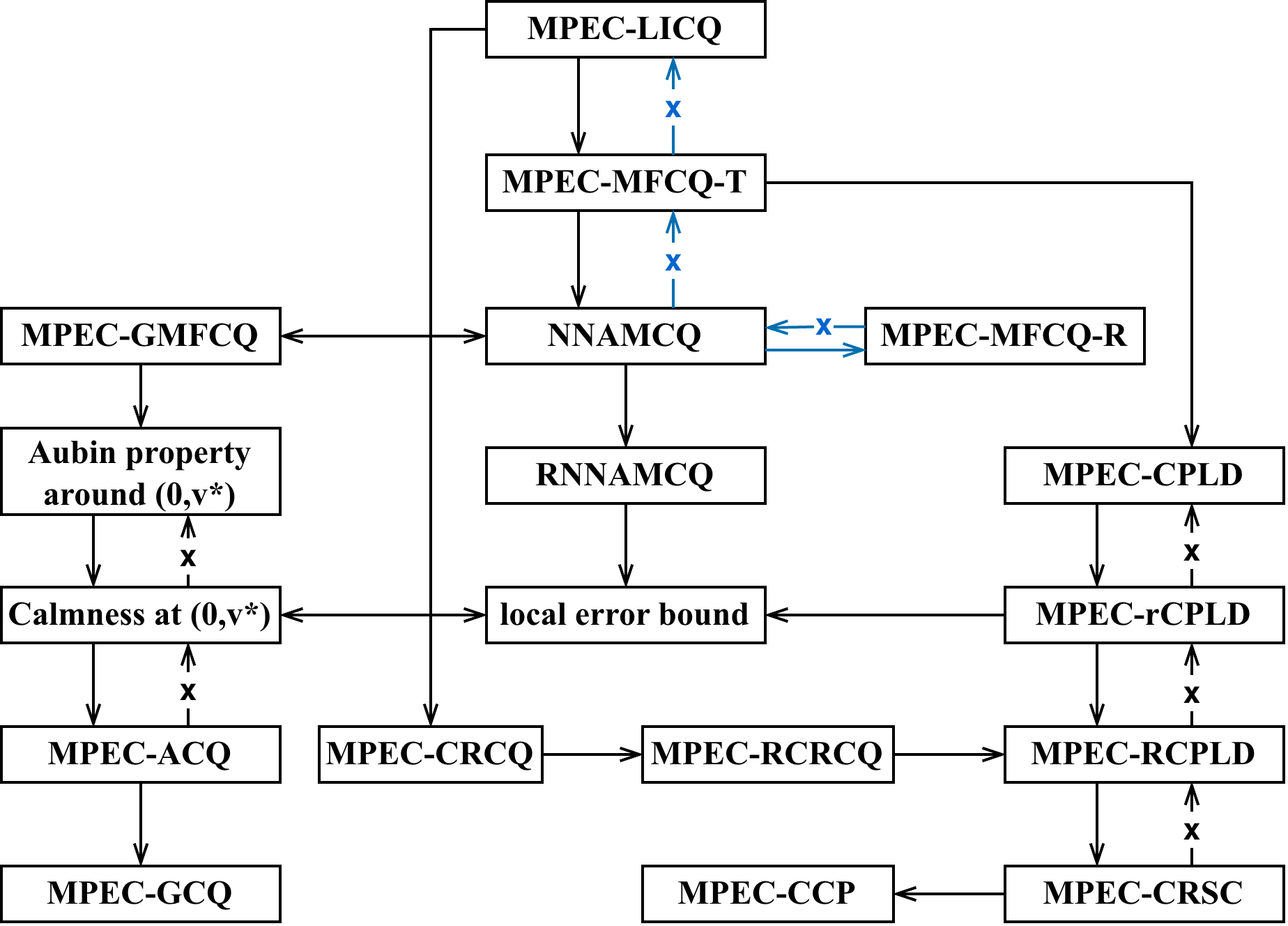}
        \caption{Relations among various MPEC-CQs}
        \label{cqsrelation}
    \end{figure}
     {
    \subsection{Discussions on Stationarity and Solution Methods}
    }
    MPEC-CQs are closely related to stationary points and solution methods for MPEC. In this part, we further discuss them.
    
    Unlike classical NLP, where suitable CQs make KKT conditions necessary, MPECs demand refined stationarity concepts to accommodate equilibrium constraints.
	\begin{definition}
		Let  $v^*$  be feasible for  \eqref{mpec0}.  $v^*$  is said to be\\
		\begin{itemize}
			\item[$\left(\mathrm{a}\right)$] weakly stationary, if there are multipliers  $\lambda \in \mathbb{R}^{m}$, $\mu \in \mathbb{R}^{p}$, $\gamma,\ \eta \in \mathbb{R}^{l}$  such that
			\begin{equation*}
				\begin{aligned}
					\nabla f\left(v^{*}\right) & +\sum_{i=1}^{m} \lambda_{i} \nabla g_{i}\left(v^{*}\right)+\sum_{i=1}^{p} \mu_{i} \nabla h_{i}\left(v^{*}\right)-\sum_{i=1}^{l} \gamma_{i} \nabla G_{i}\left(v^{*}\right)  -\sum_{i=1}^{l} \eta_{i} \nabla H_{i}\left(v^{*}\right)=0
				\end{aligned}
			\end{equation*}
			
			and  $\lambda_{i} \geq 0$, $\lambda_{i} g_{i}\left(v^{*}\right)=0\ (i=1, \ldots, m)$, $\gamma_{i}=0\ \left(i \in I_{H}\right)$, $\eta_{i}=0\ \left(i \in I_{G}\right) $;
			\item[$\left(\mathrm{b}\right)$]  C-stationary, if it is weakly stationary and  $\gamma_{i} \eta_{i} \geq 0$  for all  $i \in I_{GH} $;
			\item[$\left(\mathrm{c}\right)$] M-stationary, if it is weakly stationary and either  $\gamma_{i}>0$, $\eta_{i}>0$  or  $\gamma_{i} \eta_{i}=0$  for all  $i \in I_{GH} $;
			\item[$\left(\mathrm{d}\right)$] S-stationary if it is weakly stationary and  $\gamma_{i}, \eta_{i} \geq 0 $ for all  $i \in I_{GH} $.
		\end{itemize}
		
	\end{definition}
   
	It has been proven that S-stationary is equivalent to the standard KKT conditions of an MPEC when the MPEC is treated as a standard NLP. Moreover, S-stationary implies M-stationarity, M-stationarity implies C-stationarity, and C-stationarity further implies weak stationarity  \cite{MFL3}. However, even for very simple MPECs, S-stationary may fail to hold at a global minimum \cite{Scheel2000}. Therefore, under suitable assumptions, M-stationarity is generally the strongest stationarity concept that a local minimum can satisfy. Based on the above introduction, the relationships among various stationary points are summarized in Figure \ref{fig:stationary}. Selective results about stationary points are given below.

    \begin{figure}[H]
        \centering
        \includegraphics[width=0.9\linewidth]{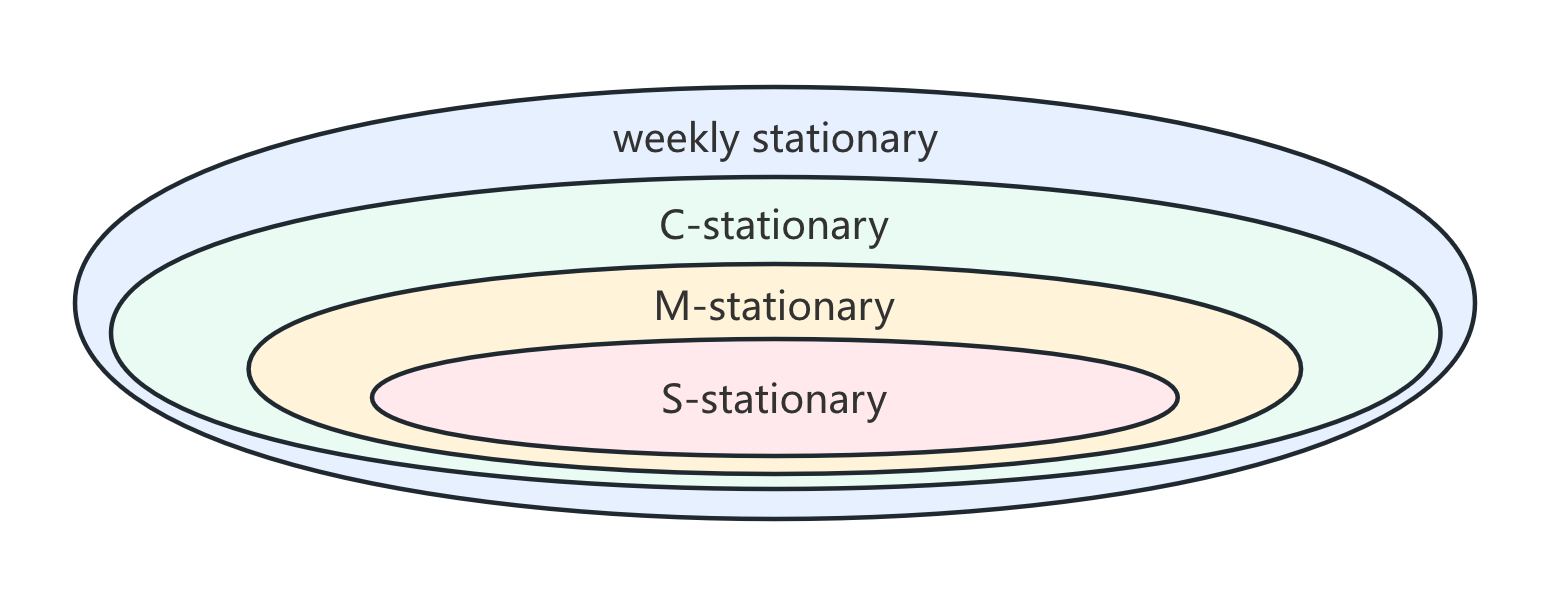}
        \caption{Various stationarities for MPECs}
        \label{fig:stationary}
    \end{figure}
	
	\begin{theorem}\textnormal{\cite{Scheel2000}}
		Let  $v^{*}$  be a local optimal solution of \eqref{mpec0}. Suppose that  the MPEC-LICQ  is satisfied at  $v^{*}$, then  $v^{*}$   is S-stationary.
	\end{theorem}
    
	\begin{theorem}\textnormal{ \cite[Theorem 3.1]{JaneYe05}} \label{acq_m_station}
		Let  $v^{*}$  be a local optimal solution of \eqref{mpec0}. Suppose that MPEC-ACQ is satisfied at  $v^{*}$, then  $v^{*}$   is M-stationary.
	\end{theorem}{ \begin{theorem}\textnormal{\cite[Corollary 3.14]{ramos}}\label{ccp_m_station}
		Let  $v^{*}$  be a local optimal solution of \eqref{mpec0}. Suppose that MPEC-CCP is satisfied at  $v^{*}$, then  $v^{*}$   is M-stationary.
	\end{theorem}
    }
    { 
    By Theorem \ref{acq_m_station} and \ref{ccp_m_station} that implies MPEC-ACQ, as well as any CQ that implies MPEC-CCP, guarantees that every local minimizer of \eqref{mpec0} is M-stationary.
    }

    { 

Appropriate MPEC-CQs can guarantee the convergence of widely used relaxation
and regularization algorithms.  We collect three such schemes into a unified
framework.

For $t>0$ and a relaxation type
$\mathcal R\in\{\mathrm{SC},\mathrm{KS},\mathrm{KDB}\}$, 
consider the smooth nonlinear program
\begin{equation}\label{R-NLP}
\begin{array}{cl}
\displaystyle \min_{v\in\mathbb R^n} & f(v)\\[1mm]
\mathrm{s.t.} &
g_i(v)\le0,\quad i\in[m],\\
& h_i(v)=0,\quad i\in[p],\\
& \mathcal R_i(v;t)\le0,\quad i\in[l],
\end{array}
\tag{$\mathrm{NLP}^{\mathcal R}(t)$}
\end{equation}
where the relaxed complementarity system is specified componentwise by
\[
\mathcal R_i(v;t):=
\begin{cases}
\bigl(-G_i(v),\,-H_i(v),\,G_i(v)H_i(v)-t\bigr),
&\text{if }\mathcal R=\mathrm{SC},\\[1mm]
\bigl(-G_i(v),\,-H_i(v),\,\Phi_i(v;t)\bigr),
&\text{if }\mathcal R=\mathrm{KS},\\[1mm]
\bigl(-G_i(v)-t,\,-H_i(v)-t,\,(G_i(v)-t)(H_i(v)-t)\bigr),
&\text{if }\mathcal R=\mathrm{KDB},
\end{cases}
\]
where $$\Phi_i(v;t)
:=
\begin{cases}
\bigl(G_i(v)-t\bigr)\bigl(H_i(v)-t\bigr),
& G_i(v)+H_i(v)\ge 2t,\\[1mm]
-\dfrac12\Bigl(\bigl(G_i(v)-t\bigr)^2+\bigl(H_i(v)-t\bigr)^2\Bigr),
& G_i(v)+H_i(v)<2t.
\end{cases}$$
With this notation,  we provide a unified
relaxation algorithm as follows, where for $\mathcal R=\mathrm{SC},\ \mathrm{KS}\  \text{and} \ \mathrm{KDB}$, it corresponds to the relaxation methods developed by Scholtes \cite{Scholtes2001}, Kanzow and Schwartz \cite{KS13,Hoheisel2013}, Kadrani, Dussault, and Benchakroun \cite{KDB09}, respectively.

\begin{algorithm}[htbp]
\caption{Relaxation method for \eqref{mpec0}}\label{alg:R-relax}
\begin{algorithmic}[1]
\Require Choose a relaxation type $\mathcal R\in\{\mathrm{SC},\mathrm{KS},\mathrm{KDB}\}$ and an initial parameter $t^0>0$.
\State Set $k \gets 0$.
\While{a stopping criterion is not satisfied}
    \State Solve problem $\mathrm{NLP}^{\mathcal R}(t^k)$ to obtain a stationary point $v^k$.
    \State Choose a new parameter $t^{k+1}$ with $0 < t^{k+1} < t^k$.
    \State Set $k \gets k+1$.
\EndWhile
\end{algorithmic}
\label{algorithm1}
\end{algorithm}

For the three concrete choices of $\mathcal R$, the following convergence
results are available. See \cite[Theorem~5]{samuel}, \cite[Theorem~5.3]{ramos} and \cite[Theorem~5.1]{ramos} for details.

\begin{theorem}\label{rKS}
    Let $\{t_k\}\downarrow 0$ and  $\{v^{k}\}$ be generated by Algorithm \ref{algorithm1}. Let $v^k$ be a stationary point of $\mathrm{NLP}^{\mathcal R}(t_k)$ with $v^{k}\to v^*$. The following results hold:
    \begin{itemize}
         \item[$\left(\mathrm{i}\right)$] For $\mathcal R=\mathrm{SC}$, if the MPEC-MFCQ-R holds at $v^*$, then $v^*$ is a C-stationary point of \eqref{mpec0}.
         \item[$\left(\mathrm{ii}\right)$]  For $\mathcal R=\mathrm{KS}\ \text{or}\ \mathrm{KDB}$, if the MPEC-CCP holds at $v^*$, then $v^*$ is an M-stationary point of \eqref{mpec0}.
    \end{itemize}
\end{theorem}

In the next section, we take the bilevel hyperparameter optimization model for
L1-loss SVC as a practical application to illustrate how such CQs can be
verified and how the corresponding relaxed problems can be solved in practice.

}

	\section{Applications to Bilevel Hyperparameter Optimization}\label{sec4}
	
    In this section, we briefly review the BHO for L1-loss SVC and its corresponding MPEC reformulation \cite{ref1}, and then analyze the relevant CQs for such MPEC.

    { 
    Let $\Omega=\{(x_i,y_i)\}_{i=1}^{l_1}\subset\mathbb{R}^{n}\times\{\pm1\}$ be the dataset used for cross–validation. In $T$–fold cross–validation, $\Omega$ is partitioned (equally) into $T$ disjoint subsets. At the $t$-th iteration ($t=1,\dots,T$), the $t$-th subset $\Omega_t$ serves as the validation set and the remaining data $\bar\Omega_t:=\Omega\setminus\Omega_t$ serve as the training set. Denote $m_1:=|\Omega_t|$ and $m_2:=|\bar\Omega_t|$. We use index sets
    $\mathcal{V}_t\subset\{1,\dots,l_1\}\quad\text{and}\quad\mathcal{N}_t\subset\{1,\dots,l_1\}$ for the validation and training points in fold $t$, respectively.

    For each fold $t$, we train a soft-margin SVC on the training set $\bar\Omega_t$ with a shared hyperparameter $C\ge0$. Using the standard $l_1$ hinge loss, the training problem is to look for a hyperplane $x^\top w^t=0$ such that the hyperplane can separate the data $\bar\Omega_t$ as well as possible. That is, $w^t$ is the solution of the following problem 
    \begin{equation}
    \label{eq:train}
    w^t\in\arg\min_{w\in\mathbb{R}^n}\;
    \frac12\|w\|_2^2
    + C\sum_{i\in\mathcal{N}_t}\bigl(1-y_i x_i^\top w\bigr)_+,
    \end{equation}
    where $(\cdot)_+:=\max\{\cdot,0\}$.

    Given the trained $w^t$ in \eqref{eq:train}, a validation point $(x_i,y_i)$ is misclassified if $y_i x_i^\top w^t<0$. Counting misclassified validation points can therefore be written as
    \[
    \sum_{i\in\mathcal{V}_t}\bigl\|\,\bigl(-y_i x_i^\top w^t\bigr)_+\,\bigr\|_0,
    \]
    where $\|\cdot\|_0$ denotes the ``zero-norm'' that returns $1$ for a positive argument and $0$ otherwise. 
    }
    
     {Equivalently, each summand is an indicator that equals $1$ if $y_i(x_i^\top w^t)<0$ and $0$ if $y_i(x_i^\top w^t)\ge 0$, so the objective computes the average misclassification count across folds.} It is evident that  $\left\|(\cdot)_{+}\right\|_{0}$  is both discontinuous and nonconvex. However, as demonstrated by Mangasarian \cite{Mangasarian1994}, this function can be expressed as the minimum of the sum of all elements in the solution to the following linear optimization problem
	\begin{equation}\label{r0}
		\left\|r_{+}\right\|_{0}=\left\{\min \sum_{i=1}^{m_{1}} \zeta_{i} \ | \  \zeta=\underset{u}{\operatorname{argmin}}\left\{-u^{\top} r: \mathbf{0} \leq u \leq \mathbf{1}\right\}\right\} .
	\end{equation}

    In the special case where \( r_i = 0 \), the function \( \|(r_i)_+\|_0 \) evaluates to 0. However, the solution to the corresponding linear programming problem is not unique and can take any value within the interval \([0,1]\). To maintain the robustness and consistency of formulation \eqref{r0}, we assume throughout this paper that the SVC model assigns a distinct classification to all validation data points. More precisely, for each validation data point \( x_i \), where \( i \in [m_1] \), we impose the condition  $x_i^\top w^t \neq 0$, where \( w^t \) is the solution for the lower level problem in \eqref{eq31}, $t \in [T]$.
    
    { 
Putting the validation objective and the foldwise training subproblems together yields the bilevel program that selects the hyperparameter $C$:
\begin{equation}\label{eq31}
		\begin{aligned}
			\min \limits _{C \in \mathbb{R}, \ w^{t} \in \mathbb{R}^{n}, \  t=1,  \cdots, T} & \frac{1}{T}\sum_{t=1}^{T} \frac{1}{m_{1}} \sum_{i \in \mathcal{V}_{t}} \Bigl\| \left( -y_{i} \left( x_{i}^{\top} w^{t}\right) \right)_{+}\Bigr\|_{0}\\
			\hbox{s.t.} \quad \quad \quad \ & C \geq 0, \\
			\quad \quad \quad \ & \text{and for} \quad t=1, \cdots, T: \\
			\quad \quad & w^{t} \in \underset{ w \in \mathbb{R}^{n}}{\operatorname{argmin}} \left\{\frac{1}{2}\|w\|_{2}^{2}+ C \sum_{i \in {\mathcal{N}}_{t}}  \left( 1-y_{i}\left(x_{i}^{\top}w\right)\right)_{+}
			\right\}.
		\end{aligned}	
	\end{equation} 
In other words, the upper level chooses $C$ to minimize the average (across folds) validation misclassification count, while the lower level trains, for each fold, a standard soft–margin SVC with that common $C$ on the corresponding training subset.

}

    Let $Q_u=\{1,2,\ldots,Tm_1\}$ and $Q_l=\{1,2,\ldots,Tm_2\}$ be the stacked index sets for the validation and training samples, respectively. Replacing the lower level with its KKT conditions and using the reformulation of $\|\cdot\|_0$,  problem \eqref{eq31} can be reformulated  as follows
	\begin{equation} \label{eq15}
		\begin{aligned}
			& \min_{\begin{subarray}{c}
					C \in \mathbb{R} \\
					\zeta \in \mathbb{R}^{Tm_{1}}, \ z \in \mathbb{R}^{Tm_{1}}\\
					\alpha \in \mathbb{R}^{Tm_{2}}, \ \xi \in \mathbb{R}^{Tm_{2}}
			\end{subarray}}  \frac{1}{Tm_{1}} \mathbf{1}^{\top} \zeta\\
			& \quad  \quad \ \begin{array}{ll}\hbox{s.t.}\ \ \mathbf{0} \leq  \zeta \perp  A B^{\top} \alpha+z  \geq \mathbf{0},  \\
				\quad \quad \mathbf{0} \leq  z  \perp  \mathbf{1}-\zeta  \geq \mathbf{0},   \\
				\quad \quad \mathbf{0} \leq  \alpha  \perp  B B^{\top} \alpha- \mathbf{1}+\xi \geq \mathbf{0}, \\
				\quad \quad \mathbf{0} \leq  \xi  \perp  C \mathbf{1}-\alpha  \geq \mathbf{0}, 
			\end{array}
		\end{aligned}
	\end{equation}
	where $\zeta \in \mathbb{R}^{Tm_1}, \ z \in \mathbb{R}^{Tm_1}, \ A \in \mathbb{R}^{Tm_{1} \times {Tn}}, \ \alpha \in \mathbb{R}^{Tm_{2}}, \ \xi \in \mathbb{R}^{Tm_{2}}$,  and $B \in \mathbb{R}^{Tm_{2} \times Tn}$ are defined by
	$$
	%\be \label{eqa2}
	\setlength{\abovedisplayskip}{3pt}
	\begin{array}{c}
		\zeta\!:=\!\left[\begin{array}{l} \zeta^{1} \\ \zeta^{2} \\ \vdots \\ \zeta^T
		\end{array}\right] , \ 
		z \!:=\!\left[\begin{array}{l} z^{1} \\ z^{2} \\ \vdots \\ z^T
		\end{array}\right] , \
		A:=\left[\begin{array}{cccc} A^{1} & \mathbf{0} & \cdots & \mathbf{0} \\ \mathbf{0} &A^{2} & \cdots& \mathbf{0} \\ \vdots &\vdots& \ddots & \vdots \\
			\mathbf{0} & \mathbf{0} & \cdots & A^T
		\end{array}\right], 
	\end{array}
	%\ee
	$$
	
	\begin{equation*} \label{eqvar_lower}
		\alpha:=\left[\begin{array}{l} \alpha^{1} \\ \alpha^{2}\\ \vdots \\ \alpha^T
		\end{array}\right], \ 
		\xi:=\left[\begin{array}{l} \xi^{1} \\ \xi^{2} \\ \vdots \\ \xi^T
		\end{array}\right], \ 
		  B:=\left[\begin{array}{cccc} B^{1} & \mathbf{0} & \cdots & \mathbf{0} \\ \mathbf{0} &B^{2} & \cdots& \mathbf{0} \\ \vdots &\vdots& \ddots & \vdots \\
			\mathbf{0} & \mathbf{0} & \cdots & B^T
		\end{array}\right], 
	\end{equation*}
	where 
	\[
        \begin{aligned}
        A^{t}&=\left[\begin{array}{c}
        y_{t_1} x_{t_1}^{\top} \\
        \vdots \\
        y_{t_{m_{1}}} x_{t_{m_{1}}}^{\top}
        \end{array}\right] \in \mathbb{R}^{m_{1} \times n}, \quad (x_{t_k}, y_{t_k}) \in \Omega_{t}, \\
        B^{t}&=\left[\begin{array}{c}
        y_{t_{m_{1}+1}} x_{t_{m_{1}+1}}^{\top} \\
        \vdots \\
        y_{t_{l_{1}}} x_{t_{l_{1}}}^{\top}
        \end{array}\right] \in \mathbb{R}^{m_{2} \times n}, \quad (x_{t_k}, y_{t_k}) \in \overline{\Omega}_{t}.
        \end{aligned}
    \]
    { {A detailed conversion of the above reformulation by replacing the lower-level problems with their KKT conditions and by using the standard complementarity reformulation of the $l_0$ term is provided in Appendix \ref{appendixa}.}}
    
	Problem \eqref{eq15} can be written in a compact form as follows:
    % (See  \cite{ref1} for more details of how the model was derived)
	\begin{equation} \label{mpec}
		\begin{aligned}
			& \min_{v \in \mathbb{R}^{d}} f(v)\\
			& \quad \hbox{s.t.}\ \  \mathbf{0} \leq  H(v)  \perp  G(v)  \geq \mathbf{0}.
		\end{aligned}
        \tag{BHO-SVM}
	\end{equation}
	where $v = \left[C, \ \zeta^{\top}, \ z^{\top}, \ \alpha^{\top}, \ \xi^{\top}\right]^{\top}  \in  \mathbb{R}^{d}$
	with $d= 2 T (m_{1}+m_{2})+1$
	and  the functions $f:\mathbb{R}^{d} \to \mathbb{R}$,  $G:\mathbb{R}^{d} \to \mathbb{R}^{d-1}$,  $H:\mathbb{R}^{d} \to \mathbb{R}^{d-1}$ take the following form
	\begin{equation}\label{eqe1}
		f(v) =  c^{\top}v,    \ H(v)  =  \mathcal{Q}v,\ \mbox{ and }\ \ G(v) =  \mathcal{P}v+a,
	\end{equation}
	where
	$$
	c\!=\!\frac{1}{T m_{1}} \left[\begin{array}{c} 0\\ \mathbf{1}_{T m_1} \\ \mathbf{0}_{T m_1 } \\ \mathbf{0}_{T m_2 } \\ \mathbf{0}_{T m_2 } \end{array}\right] \! \in \! \mathbb{R}^{d}, \ 
	a \! = \! \left[\begin{array}{c} \mathbf{0}_{T m_1} \\ \mathbf{1}_{T m_1} \\-\mathbf{1}_{T m_2} \\ \mathbf{0}_{T m_2} \end{array}\right] \! \in \! \mathbb{R}^{d-1}, \ \mathcal{Q} \!= \!\left[\begin{array}{cc} \mathbf{0}_{d-1}& I_{d-1} \end{array}\right] \! \in \! \mathbb{R}^{(d-1) \times d}, 
	$$
	$$ \mathcal{P} \!= \!\left[\begin{array}{ccccc}\mathbf{0}_{T m_1 }& \mathbf{0}_{T m_1 \times T m_1} & I_{T m_1} & A B^{\top} & \mathbf{0}_{T m_1 \times T m_2} \\ \mathbf{0}_{T m_1} &-I_{T m_1 } & \mathbf{0}_{T m_1 \times T m_1} & \mathbf{0}_{T m_1 \times T m_2} & \mathbf{0}_{T m_1 \times T m_2}\\ \mathbf{0}_{T m_2} &\mathbf{0}_{T m_2 \times T m_1} & \mathbf{0}_{T m_2 \times T m_1} & B B^{\top} & I_{ T m_2} \\
		\mathbf{1}_{T m_2 } &\mathbf{0}_{T m_2 \times T m_1} & \mathbf{0}_{T m_2 \times T m_1} & -I_{T m_2 } & \mathbf{0}_{T m_2 \times T m_2 }\end{array}\right]  \! \in \! \mathbb{R}^{(d-1) \times d}. $$
	
	Compared with \eqref{mpec0}, the MPEC in \eqref{mpec} has the following properties.

	\begin{theorem}\label{acqequivalent}
		For \eqref{mpec}, the following results hold:
        \begin{itemize}
            {  
            \item[$\left(\mathrm{i}\right)$]  The constraint system of \eqref{mpec} satisfies  the calmness at all feasible points.
            \item [$\left(\mathrm{ii}\right)$]   MPEC-RCRCQ holds at all feasible points.
            \item [$\left(\mathrm{iii}\right)$] MPEC-CPLD holds at all feasible points.
            }
			\item[$\left(\mathrm{iv}\right)$]   MPEC-MFCQ-T holds at a feasible point $v^*$ if and only if    MPEC-LICQ holds at $v^*$.
		\end{itemize}
	\end{theorem}

    The proof of Theorem \ref{acqequivalent} is provided in Appendix \ref{appendixb}.
    { 
    \begin{remark}
        The assertions in Theorem~\ref{acqequivalent} are not restricted to problem~\eqref{mpec}.
        More precisely, if all constraint functions are affine, then
        items~$\left(\mathrm{i}\right)$ and $\left(\mathrm{ii}\right)$ hold automatically at every feasible point.
        For any MPEC that contains only complementarity constraints, items 
        $\left(\mathrm{iv}\right)$ remains valid.
    \end{remark}
	}

    By Theorem \ref{acqequivalent}, we know that, for problem \eqref{mpec}, MPEC-MFCQ-T coincides with the strongest classical CQ, MPEC-LICQ. Next,  we will mainly discuss the fulfillment of  MPEC-MFCQ-R and MPEC-LICQ for  $\eqref{mpec}$.

    \begin{theorem}\label{aplmfcqr}
    Let  $v^*$  be a feasible point of problem \eqref{mpec}. If $\left(B B^{\top}\right)_{\left(\Lambda_{1} \cup \Lambda_{3}, \Lambda_{1} \cup \Lambda_{3}\right)}$ is positive definite at $v^*$, then  the MPEC-MFCQ-R automatically holds at $ v^* $, where $\Lambda_1$ and  $\Lambda_3$ are defined by
    \begin{equation*}
        \begin{aligned}
			\Lambda_{1}&:=\left\{i \in Q_{l}\ \mid \ \alpha_{i}=0, \ (B B^{\top} \alpha-\mathbf{1}+\xi)_{i}=0, \ \xi_{i}=0\right\},  \\
			\Lambda_{3} &:=\left\{i \in Q_{l}\ \mid \ 0<\alpha_{i}\leq C, \ (B B^{\top} \alpha-\mathbf{1}+\xi)_{i}=0, \ \xi_{i}=0\right\}.
        \end{aligned}
    \end{equation*}
    \end{theorem}
    The proof for Theorem \ref{aplmfcqr} is included in Appendix \ref{appendixc}.

    To state the MPEC-LICQ result, we introduce the index sets
    \begin{equation*} \label{eq23}
		\begin{aligned}
			& I_{GH_{3}} :=\left\{i \in Q_{l} \ \mid \ \alpha_{i}=0, \ (BB^{\top} \alpha -\mathbf{1}+\xi)_{i}=0\right\},  \\
			& I_{GH_{4}} :=\left\{i \in Q_{l} \ \mid \ \xi_{i}=0, \ C-\alpha_{i}=0\right\},\\
            &\Lambda^{+}_{3} :=\left\{i \in Q_{l}\ \mid \ 0<\alpha_{i}<C, \ (B B^{\top} \alpha-\mathbf{1}+\xi)_{i}=0, \ \xi_{i}=0\right\},   \\
			&\Lambda^{c}_{3} := \left\{i \in Q_{l}\ \mid \ \alpha_{i}=C, \ (B B^{\top} \alpha-\mathbf{1}+\xi)_{i}=0, \ \xi_{i}=0\right\},  \\
			&\Lambda_{u}:=\left\{i \in Q_{l}\ \mid \ \alpha_{i}=C, \ (B B^{\top} \alpha-\mathbf{1}+\xi)_{i}=0, \ \xi_{i}>0\right\}.
		\end{aligned}
	\end{equation*}

    { 
    \begin{theorem}\label{LICQ_new}
		Let  $v^*$  be a feasible point of problem \eqref{mpec}.
		\begin{itemize}
			\item[$\left(\mathrm{i}\right)$] If  $\left | I_{GH} \right | > 1 $, then  the MPEC-LICQ fails at $ v^* $.
			\item[$\left(\mathrm{ii}\right)$] If  $\left | I_{GH} \right | =0 $, then  the MPEC-LICQ holds at $ v^* $.
			\item[$\left(\mathrm{iii}\right)$] If $\left |  I_{GH}\right |=1$,   $(BB^{\top})_{(\Lambda_{3}^+, \Lambda_3^+)}$ is positive definite at $v^*$ and $\Gamma_{sub} \neq 0$, then  the MPEC-LICQ holds at $ v^* $, where
            \begin{equation*}
                \Gamma_{\mathrm{sub}}=\begin{bmatrix}(BB^\top)_{(I_{GH_4},\Lambda_{3}^c\cup\Lambda_{u})}\,\mathbf 1
 - (BB^\top)_{(I_{GH_4},\Lambda_{3}^+)}\,
   (BB^\top)_{(\Lambda_{3}^+,\Lambda_{3}^+)}^{-1}\,
   (BB^\top)_{(\Lambda_{3}^+,\Lambda_{3}^c\cup\Lambda_{u})}\,\mathbf 1\\
(BB^\top)_{(I_{GH_3},\Lambda_{3}^c\cup\Lambda_{u})}\,\mathbf 1
 - (BB^\top)_{(I_{GH_3},\Lambda_{3}^+)}\,
   (BB^\top)_{(\Lambda_{3}^+,\Lambda_{3}^+)}^{-1}\,
   (BB^\top)_{(\Lambda_{3}^+,\Lambda_{3}^c\cup\Lambda_{u})}\,\mathbf 1
                \end{bmatrix}.
            \end{equation*}

		\end{itemize}
	\end{theorem}
    
    We provide the complete proof of Theorem \ref{LICQ_new} in Appendix \ref{appendixd}.    
    }
    \begin{remark}
		Note that  $\left | I_{GH} \right | =0$ indicates that strict complementarity  conditions hold at $v^*$.  
		From the numerical point of view,  given a specific feasible point $v^*$ of $\eqref{mpec}$,  it may be very often that strict complementarity  conditions hold at $v^*$,  since it may be very rare that numerically two terms $H_i(v)$ and $G_i(v)$ are zero at the same time.
	\end{remark}

    { 
    Figure \ref{cqrelationforsvm} summarizes the  relations among the constraint qualifications (CQs) used in this paper for the MPEC model \eqref{mpec}. Since calmness, MPEC-CPLD, and MPEC-RCRCQ have been shown to hold at all feasible points of (BHO-SVM), any CQ implied by these properties is automatically satisfied; hence the corresponding implied branches in Figure \ref{cqsrelation} are omitted in Figure \ref{cqrelationforsvm} for clarity. 
    }

\begin{figure}[H]
\centering
\begin{tikzpicture}[
  >=Stealth, font=\small,
  cq/.style={
    draw, rounded corners, thick, align=center,
    minimum width=3.6cm, minimum height=1.05cm, inner sep=4pt
  },
  arr/.style={thick, shorten >=2pt, shorten <=2pt}
]

\node[cq] (mfcqt) {\textbf{MPEC-MFCQ-T}};
\node[cq, above=1.55cm of mfcqt] (licq)
  {\textbf{MPEC-LICQ}\\ \footnotesize Theorem \ref{LICQ_new}};

\node[cq, right=3.2cm of licq] (rcrcq)
  {\textbf{MPEC-RCRCQ}\\ \footnotesize Theorem \ref{acqequivalent}. $(ii)$};

\node[cq, right=3.2cm of mfcqt] (cpld)
  {\textbf{MPEC-CPLD}\\ \footnotesize Theorem \ref{acqequivalent}. $(iii)$};

\node[cq, above=0.25cm of cpld] (calm)
  {\textbf{Calmness}\\ \footnotesize Theorem \ref{acqequivalent}. $(i)$};

\node[cq, below=0.25cm of cpld] (mfcqr)
  {\textbf{MPEC-MFCQ-R}\\ \footnotesize Theorem \ref{aplmfcqr}};

\draw[arr,<->] (licq.south) -- node[right]{\scriptsize Theorem \ref{acqequivalent}. $(iv)$} (mfcqt.north);
\draw[arr,->]  (licq.east) -- (rcrcq.west);

\path (mfcqt.east) ++(1.25,0) coordinate (branch);

\draw[thick] (mfcqt.east) -- (branch);

\draw[thick,->,shorten <=0pt,shorten >=2pt] (branch) |- (calm.west);
\draw[thick,->,shorten <=0pt,shorten >=2pt] (branch) |- (cpld.west);
\draw[thick,->,shorten <=0pt,shorten >=2pt] (branch) |- (mfcqr.west);

\end{tikzpicture}
\caption{Constraint qualifications for \eqref{mpec}}
\label{cqrelationforsvm}
\end{figure}

{ 
To summarize, for \eqref{mpec}, calmness property, MPEC-CPLD  and MPEC-RCRCQ hold automatically at all feasible points. Under certain conditions, MPEC-LICQ, MPEC-MFCQ-T, and MPEC-MFCQ-R hold. These results provide guidance for selecting solution methods. For example, one can use the relaxation algorithm of Kanzow and Schwartz to solve \eqref{mpec}, and the convergence conclusion in Theorem~\ref{rKS} holds automatically, since MPEC-CPLD implies MPEC-CCP, MPEC-CCP also holds at all feasible points.
}

	\section{Conclusions}\label{sec-6}
	In this paper, { we study several types of constraint qualifications (such as MPEC-LICQ, MPEC-MFCQ-T, NNAMCQ, MPEC-MFCQ-R etc.) developed for MPECs, and provide a thorough theoretical analysis of their relationships.} Then we address the hyperparameter selection problem for support vector classification,  reformulating it as a mathematical program with equilibrium constraints (MPEC) and analyzing the relevant constraint qualifications (CQs). The study highlights the role of positive linear independence and demonstrates the equivalence of  the MPEC-MFCQ-T and  the MPEC-LICQ for this problem. Furthermore, we establish refined conditions under which  the MPEC-LICQ holds and identify specific cases where it may fail.  The core contribution lies in providing a more robust theoretical foundation for complex bilevel optimization models,  especially in the context of hyperparameter optimization in modern machine learning.
	Future research could explore the extension of the applicability of MPEC-related CQs and further investigate their properties in more complex models.

    \vskip 1cm
	Disclosure statement:
	There are no relevant financial or non-financial competing interests to report.

	\appendix\label{sec-5}
    { 
    \section{Transformation of the Bilevel Problem into the MPEC Formulation}\label{appendixa}
We aim to select the optimal hyperparameter $C$ for soft-margin support vector classification (SVC) using $T$-fold cross-validation. The resulting model is formulated as a bilevel optimization problem.

For each fold $t = 1, \dots, T$, the lower-level problem is a soft-margin SVC trained on the training set $\bar{\Omega}_t$. The optimization problem is
\begin{equation*}
w^t \in \arg\min_{w \in \mathbb{R}^n} \left\{ \frac{1}{2} \|w\|_2^2 + C \sum_{i \in \mathcal{N}_t} \left(1 - y_i (x_i^\top w)\right)_+ \right\}.
\end{equation*}

This problem can be reformulated with slack variables and KKT conditions. Letting $B^t$ represent the training data matrix, and $\xi^t, \alpha^t, \mu^t$ the corresponding slack and dual variables, we obtain the reduced KKT system:
\begin{align*}
&0 \leq \alpha^t \perp B^t w^t - 1 + \xi^t \geq 0, \\
&0 \leq \xi^t \perp C \mathbf{1} - \alpha^t \geq 0, \\
&w^t = (B^t)^\top \alpha^t.
\end{align*}

The upper-level goal is to minimize the cross-validation (CV) error on the validation sets $\Omega_t$ using the trained models $w^t$. The number of misclassified points is measured by the following function
\begin{equation*}
\min_{C \in \mathbb{R}} \frac{1}{T} \sum_{t=1}^\top \frac{1}{m_1} \left\| (-A^t w^t)_+ \right\|_0.
\end{equation*}

Here, $A^t$ is the matrix formed by validation data, and $\|\cdot\|_0$ counts misclassified points. Since the $l_0$ norm is discontinuous and nonconvex, it is reformulated using an auxiliary  linear programming:
\begin{equation*}
    \begin{aligned}
    \min_{\zeta^t \in \mathbb{R}^{m_1}} \quad & \sum_{i=1}^{m_1} \zeta^t_i \\
    \text{s.t.} \quad & 0 \leq \zeta^t \leq 1, \\
    & \zeta^t_i = 1 \text{ if } (x_i, y_i) \text{ is misclassified}, \\
    & \zeta^t_i = 0 \text{ otherwise}.
    \end{aligned}
\end{equation*}

This linear programming can be equivalently expressed via its KKT conditions with auxiliary dual variables $\lambda^t, z^t$
\begin{align*}
0 \leq \zeta^t &\perp A^t w^t - \lambda^t + z^t \geq 0, \\
0 \leq z^t &\perp 1 - \zeta^t \geq 0.
\end{align*}

Substituting $w^t = (B^t)^\top \alpha^t$ as in the lower-level KKT conditions, the entire bilevel problem can be flattened into a single-level Mathematical Program with Equilibrium Constraints (MPEC):
\begin{equation}\label{ampec}
\begin{aligned}
\min_{C \in \mathbb{R},\, \zeta,\, z,\, \alpha,\, \xi} \quad & \frac{1}{T m_1} \mathbf{1}^\top \zeta, \\
\text{s.t.} \quad
& 0 \leq \zeta \perp A B^\top \alpha + z \geq 0, \\
& 0 \leq z \perp \mathbf{1} - \zeta \geq 0, \\
& 0 \leq \alpha \perp B B^\top \alpha - 1 + \xi \geq 0, \\
& 0 \leq \xi \perp C \mathbf{1} - \alpha \geq 0.
\end{aligned}
\end{equation}

Here, the matrices $A$ and $B$ are block-diagonal aggregations of all $A^t$ and $B^t$, respectively, and $\zeta, z, \alpha, \xi$ are concatenated vectors over all folds $t = 1, \dots, T$.

\eqref{ampec} is the MPEC  version of the original bilevel problem, suitable for further theoretical analysis or numerical solution.}

    \section{Proof of Theorem \ref{acqequivalent}}\label{appendixb}
    \begin{proof}
			{   $\left(\mathrm{i}\right)$ By Remark \ref{affinecalmness},   the constraint system of \eqref{mpec} satisfies calmness at all the feasible point, since all functions in \eqref{mpec} are affine linear.}
            
        \vspace{2.5mm}
        {  \noindent $\left(\mathrm{ii}\right)$ 
        When all constraint functions are affine functions of $v$, the MPEC–RCRCQ condition is automatically satisfied at every feasible point, since all gradients are constant and the ranks of all relevant gradient families do not change in a neighborhood.}

        \vspace{2.5mm}
    {  \noindent $\left(\mathrm{iii}\right)$ Fix an arbitrary feasible point $v^*$ of \eqref{mpec}.  
    Since the problem contains no functions $g$ and $h$, by Definition \ref{def_MPECCPLD}, the positively linearly dependence of \eqref{def-mpec-cpld} reduces to the linearly dependence of the following  set of gradients 
\[
\bigl\{\nabla G_i(v^*) \mid i\in I_3\bigr\} \cup\bigl\{\nabla H_i(v^*) \mid i\in I_4\bigr\},
\ 
I_3\subseteq I_{GH}\cup I_G,\ I_4\subseteq I_{GH}\cup I_H.
\]

Then there exist scalars
$\{\lambda_i\}_{i\in I_3}$ and $\{\mu_i\}_{i\in I_4}$, not all zero, such that
\begin{equation}\label{eq:cpld-lindep}
\sum_{i\in I_3}\lambda_i \nabla G_i(v^*) \;+\; \sum_{i\in I_4}\mu_i \nabla H_i(v^*) \;=\; 0.
\end{equation}
Since $G$ and $H$ are affine, their Jacobians are constant, i.e.,
$\nabla G_i(v)=\nabla G_i(v^*)$ and $\nabla H_i(v)=\nabla H_i(v^*)$ for all $v$. Therefore, the identity \eqref{eq:cpld-lindep} holds
for every $v$ in a neighbourhood of $v^*$, namely,
\[
\sum_{i\in I_3}\lambda_i \nabla G_i(v) \;+\; \sum_{i\in I_4}\mu_i \nabla H_i(v) \;=\; 0.
\]

Since the choice of $v^*$ was arbitrary, MPEC-CPLD
holds at every feasible point.
    }

        \vspace{2.5mm}
	\noindent $\left(\mathrm{iv}\right)$ Since \eqref{mpec} does not have inequality constraints  $g $, according to Definition \ref{def-mfcq0} and Definition \ref{def-licq0},  the MPEC-LICQ and  the MPEC-MFCQ-T are equivalent for \eqref{mpec} at all the feasible point.
\end{proof}

	\section{Proof of Theorem \ref{aplmfcqr}}\label{appendixc}
	
    Recall $I_{H}:=\underset{k=1} {\overset{4}{\cup}}I_{H_{k}}$, $ \ I_{G}:=\underset{k=1} {\overset{4}{\cup}}I_{G_{k}} $ and $I_{GH}:=\underset{k=1} {\overset{4}{\cup}}I_{GH_{k}}$ in \cite{ref1}, where
	\begin{subequations} \label{eq231}
		\begin{align*}
			& I_{H_{1}} :=  \left\{i \in Q_{u}  \mid  \zeta_{i}=0, \ (AB^{\top} \alpha +z)_{i}>0\right\},  \ I_{H_{2}} := \left\{i \in Q_{u}  \mid  z_{i}=0, \ 1-\zeta_{i}>0\right\},   \\
			& I_{H_{3}}  :=  \left\{i \in Q_{l}  \mid  \alpha_{i}=0, \ (BB^{\top} \alpha -\mathbf{1}+\xi)_{i}>0\right\},  \ I_{H_{4}}  :=  \left\{i \in Q_{l}  \mid  \xi_{i}=0, \ C-\alpha_{i}>0\right\},  \\
			& I_{G_{1}}:=\left\{i \in Q_{u}  \mid \zeta_{i}>0, \ (AB^{\top} \alpha +z)_{i}=0\right\}, \ I_{G_{2}}  :=\left\{i \in Q_{u}  \mid  z_{i}>0, \ 1-\zeta_{i}=0\right\}, \\
			& I_{G_{3}} := \left\{i \in Q_{l}  \mid \alpha_{i}>0, \ (BB^{\top} \alpha -\mathbf{1}+\xi)_{i}=0\right\}, \ I_{G_{4}}:= \left\{i \in Q_{l} \mid \xi_{i}>0, \ C-\alpha_{i}=0\right\},\\
            & I_{G H_{1}}  :=\left\{i \in Q_{u} \mid \zeta_{i}=0,\left(A B^{\top} \alpha+z\right)_{i}=0\right\}, \ I_{G H_{2}} :=\left\{i \in Q_{u} \mid z_{i}=0,1-\zeta_{i}=0\right\}, \\
            &I_{G H_{3}}  :=\left\{i \in Q_{l} \mid \alpha_{i}=0,\left(B B^{\top} \alpha-\mathbf{1}+\xi\right)_{i}=0\right\}, \ I_{G H_{4}}  :=\left\{i \in Q_{l} \mid \xi_{i}=0, C-\alpha_{i}=0\right\} .  
		\end{align*}
	\end{subequations}

	To provide a clearer description of the assumptions,  we need  the definitions for the following index sets.
		\begin{eqnarray*}
			\Lambda_{2}&:=&\left\{i \in Q_{l} \mid  \alpha_{i}=0, \ (B B^{\top} \alpha-\mathbf{1}+\xi)_{i}>0, \ \xi_{i}=0\right\},  \label{eq6b}\\
			\Lambda_{3}&:=&\left\{i \in Q_{l} \mid  0< \alpha_{i} \leq C, \ (B B^{\top} \alpha-\mathbf{1}+\xi)_{i}=0, \ \xi_{i}=0\right\}, \label{eqLambda3}\\
			\Psi_{2}&:=&\left\{i \in Q_{u} \mid  \zeta_{i}=0, \ (A B^{\top} \alpha+z)_{i}>0, \ z_{i}=0\right\},  \label{eq7b}\\
			\Psi_{3}&:=&\left\{i \in Q_{u} \mid \zeta_{i}=1, \ (A B^{\top} \alpha+z)_{i}=0, \ z_{i}>0\right\} . \label{eq7d}
		\end{eqnarray*}
	
    \begin{proposition}\label{proIghpl} \label{pro1}
		The relationship between index sets\\
		(a) $I_{H_{1}}=\Psi_{2}, \  I_{G_{1}}= \Psi_{3}, \ I_{G H_{1}}=\emptyset.$\\
		(b)  $I_{H_{2}}=\Psi_{2}, \ I_{G_{2}}=\Psi_{3}, \  I_{G H_{2}}=\emptyset.$\\
		(c)  $I_{H_{3}}=\Lambda_{2}, \ I_{G_{3}}=\Lambda_{3} \cup \Lambda_{u}, \ I_{G H_{3}}=\Lambda_{1}.$\\
		(d) $I_{H_{4}}=\Lambda_{1} \cup \Lambda_{2} \cup \Lambda_{3}^{+}, \ I_{G_{4}}=\Lambda_{u}, \ I_{G H_{4}}=\Lambda_{3}^{c} $.
	\end{proposition}
    \begin{proof}
    Based on our discussion after \eqref{r0}, 
    $x_i^\top w^t \neq 0$ for each validation data point $ x_i$, where $ i \in [m_1]$, $t \in [T]$.
    It implies that index set $\left\{i \in Q_{u} \mid 0 \leq \zeta_{i}<1,\left(A B^{\top} \alpha+z\right)_{i}=0, z_{i}=0\right\}=\emptyset$. Likewise, the index set $\Psi_1$ in \cite{ref1} is also empty. By incorporating the content of  \cite[Proposition 5]{ref1}, the proof is complete.
    \end{proof}

	To get the  conclusions about  the MPEC-MFCQ-R and  the MPEC-LICQ at a feasible point $v^*$ for the  $\eqref{mpec}$, we need to analyze the properties of the following set of gradient vectors
	\begin{equation}\label{de3.1}
		\left\{\nabla G_{i}(v) \mid i \in I_{G} \cup I_{G H}\right\} \cup\left\{\nabla H_{i}(v) \mid i \in I_{H} \cup I_{G H}\right\}.
	\end{equation}

	\begin{proposition}\label{propab}
		  \cite[Proposition 6]{ref1} The set of gradient vectors in $\eqref{de3.1}$ at a feasible point $v^*$ for the $\eqref{mpec}$ can be written in the matrix form

        { 
        
\begin{equation}\label{gamma}
	\Gamma=
	\begin{blockarray}{*{5}{c}c}
		\scriptstyle 1 & \scriptstyle 2 & \scriptstyle 3 & \scriptstyle 4 & \scriptstyle 5 & \\[-0.2ex]
		\begin{block}{[*{5}{c}]c}
        \mathbf{0}_{(I_{G_{1}},L_{1})} & \mathbf{0}_{(I_{G_{1}},L_{2})} &
			\mathcal{I}_{(I_{G_{1}},\cdot)} & (AB^{\top})_{(I_{G_{1}},\cdot)} &
			\mathbf{0}_{(I_{G_{1}},L_{5})} & \scriptstyle 1\\
			\mathbf{0}_{(I_{H_{1}},L_{1})} & \mathcal{I}_{(I_{H_{1}},\cdot)} &
			\mathbf{0}_{(I_{H_{1}},L_{3})} & \mathbf{0}_{(I_{H_{1}},L_{4})} &
			\mathbf{0}_{(I_{H_{1}},L_{5})} & \scriptstyle 2\\
			\mathbf{0}_{(I_{G_{2}},L_{1})} & \mathcal{I}_{(I_{G_{2}},\cdot)} &
			\mathbf{0}_{(I_{G_{2}},L_{3})} & \mathbf{0}_{(I_{G_{2}},L_{4})} &
			\mathbf{0}_{(I_{G_{2}},L_{5})} & \scriptstyle 3\\
			\mathbf{0}_{(I_{H_{2}},L_{1})} & \mathbf{0}_{(I_{H_{2}},L_{2})} &
			\mathcal{I}_{(I_{H_{2}},\cdot)} & \mathbf{0}_{(I_{H_{2}},L_{4})} &
			\mathbf{0}_{(I_{H_{2}},L_{5})} & \scriptstyle 4\\
			\mathbf{0}_{(I_{G_{3}},L_{1})} & \mathbf{0}_{(I_{G_{3}},L_{2})} &
			\mathbf{0}_{(I_{G_{3}},L_{3})} & (BB^{\top})_{(I_{G_{3}},\cdot)} &
			\mathcal{I}_{(I_{G_{3}},\cdot)} & \scriptstyle 5\\
			\mathbf{0}_{(I_{GH_{3}},L_{1})} & \mathbf{0}_{(I_{GH_{3}},L_{2})} &
			\mathbf{0}_{(I_{GH_{3}},L_{3})} & (BB^{\top})_{(I_{GH_{3}},\cdot)} &
			\mathcal{I}_{(I_{GH_{3}},\cdot)} & \scriptstyle 6\\
			\mathbf{0}_{(I_{GH_{3}},L_{1})} & \mathbf{0}_{(I_{GH_{3}},L_{2})} &
			\mathbf{0}_{(I_{GH_{3}},L_{3})} & \mathcal{I}_{(I_{GH_{3}},\cdot)} &
			\mathbf{0}_{(I_{GH_{3}},L_{5})} & \scriptstyle 7\\
			\mathbf{0}_{(I_{H_{3}},L_{1})} & \mathbf{0}_{(I_{H_{3}},L_{2})} &
			\mathbf{0}_{(I_{H_{3}},L_{3})} & \mathcal{I}_{(I_{H_{3}},\cdot)} &
			\mathbf{0}_{(I_{H_{3}},L_{5})} & \scriptstyle 8\\
			\mathbf{1}_{(I_{G_{4}},L_{1})} & \mathbf{0}_{(I_{G_{4}},L_{2})} &
			\mathbf{0}_{(I_{G_{4}},L_{3})} & \mathcal{I}_{(I_{G_{4}},\cdot)} &
			\mathbf{0}_{(I_{G_{4}},L_{5})} & \scriptstyle 9\\
			\mathbf{1}_{(I_{GH_{4}},L_{1})} & \mathbf{0}_{(I_{GH_{4}},L_{2})} &
			\mathbf{0}_{(I_{GH_{4}},L_{3})} & \mathcal{I}_{(I_{GH_{4}},\cdot)} &
			\mathbf{0}_{(I_{GH_{4}},L_{5})} & \scriptstyle 10\\
			\mathbf{0}_{(I_{GH_{4}},L_{1})} & \mathbf{0}_{(I_{GH_{4}},L_{2})} &
			\mathbf{0}_{(I_{GH_{4}},L_{3})} & \mathbf{0}_{(I_{GH_{4}},L_{4})} &
			\mathcal{I}_{(I_{GH_{4}},\cdot)} & \scriptstyle 11\\
			\mathbf{0}_{(I_{H_{4}},L_{1})} & \mathbf{0}_{(I_{H_{4}},L_{2})} &
			\mathbf{0}_{(I_{H_{4}},L_{3})} & \mathbf{0}_{(I_{H_{4}},L_{4})} &
			\mathcal{I}_{(I_{H_{4}},\cdot)}  & \scriptstyle 12\\
		\end{block}
	\end{blockarray}.
\end{equation}
		where $L_q$,  $q=1, ..., 5$ are the index sets of columns corresponding to the variables $C, \zeta,  z,  \alpha$ and $\xi$,  respectively,  and $\mathcal{I}_{\left(I_{G_{1}},  \cdot \right)}= \left[\mathbf{0}_{\left(I_{G_{1}},  \Psi_{2}\right)} \ \ I_{\left(I_{G_{1}},  \Psi_{3}\right)}\right] $. All the other selector blocks $\mathcal{I}_{\left(I_{\cdot},\cdot\right)}$ are defined analogously by placing identity on the relevant active coordinate set(s) and zeros elsewhere (with signs inherited from the corresponding gradient blocks).
		}

       \end{proposition}

        \noindent\textbf{Proof of Theorem \ref{aplmfcqr}}
        \begin{proof}
            Let $\rho^{\top} \Gamma=0$ where the nonzero column vector $\rho$ is defined by 
		$\rho=\left[
		\begin{array}{c}
			\left(\rho^{1}\right)^{\top},
			\left(\rho^{2}\right)^{\top},
			\left(\rho^{3}\right)^{\top},
            \left(\rho^{4}\right)^{\top},
			\left(\rho^{5}\right)^{\top},
			\left(\rho^{6}\right)^{\top},
            \left(\rho^{7}\right)^{\top},
            \left(\rho^{8}\right)^{\top},
            \left(\rho^{9}\right)^{\top},
			\left(\rho^{10}\right)^{\top},
			\left(\rho^{11}\right)^{\top},
			\left(\rho^{12}\right)^{\top}            
		\end{array}\right]^{\top}$, with $\rho^{6},\ \rho^{7},\ \rho^{10},\ \rho^{11} \ge 0$. It gives that 
		\begin{equation*}
			0=\rho^{\top} \Gamma:=\left[\begin{array}{lllll}
				S_{1} & S_{2} & S_{3} & S_{4} & S_{5}
			\end{array}\right],
		\end{equation*}
        where
		\begin{align}
    S_1=&\left(\rho^{9}\right)^{\top}\mathbf{1}_{(\Lambda_u,L_1)} + \left(\rho^{10}\right)^{\top}\mathbf{1}_{(\Lambda_3^c,L_1)}=0, \nonumber\\
    S_2=&\left(\rho^{2}\right)^{\top}\mathcal{I}_{\left(I_{H_{1}},  \cdot \right)}+\left(\rho^{3}\right)^{\top}\mathcal{I}_{\left(I_{G_{2}},  \cdot \right)}=0, \label{1103}\\
    S_3=&\left(\rho^{1}\right)^{\top}\mathcal{I}_{\left(I_{G_{1}},  \cdot \right)}+\left(\rho^{4}\right)^{\top}\mathcal{I}_{\left(I_{H_{2}},  \cdot \right)}, \label{0618}\\
    S_4=&\left(\rho^{1}\right)^{\top}(AB^{\top})_{(I_{G_1},\cdot)} + \left(\rho^{5}\right)^{\top}(BB^{\top})_{(I_{G_3},\cdot)}  
    +\left(\rho^{6}\right)^{\top}(BB^{\top})_{(I_{GH_3},\cdot)} \nonumber\\ 
    &+\left(\rho^{7}\right)^{\top}\mathcal{I}_{\left(I_{GH_{3}},  \cdot \right)}+\left(\rho^{8}\right)^{\top}\mathcal{I}_{\left(I_{H_{3}},  \cdot \right)}
    +\left(\rho^{9}\right)^{\top}\mathcal{I}_{\left(I_{G_4},  \cdot \right)}+\left(\rho^{10}\right)^{\top}\mathcal{I}_{\left(I_{GH_{4}},  \cdot \right)}=0, \nonumber\\
    S_5=&
    \left(\rho^{5}\right)^{\top}\mathcal{I}_{\left(I_{G_{3}},  \cdot \right)}+  \left(\rho^{6}\right)^{\top}\mathcal{I}_{\left(I_{GH_{3}},  \cdot \right)}
    + \left(\rho^{11}\right)^{\top}\mathcal{I}_{\left(I_{GH_{4}},  \cdot \right)} + \left(\rho^{12}\right)^{\top}\mathcal{I}_{\left(I_{H_{4}},  \cdot \right)} =0. \label{0115}
\end{align}

To manage multiple multipliers and their associated index sets in a unified way, we use a standardized notation. Each multiplier $\rho^{k}$ is decomposed based on designated index sets, defined as 
$\rho^{k} = \bigl(\rho_{A}^{k},\ \rho_{B}^{k}\bigr)$, where $A$ and $B$ correspond to specific components of these sets. For example, the multiplier $\rho^{5}$ is decomposed as
$\rho^{5}
=\bigl(\rho^{5}_{\Lambda_{3}^{+}},\  \rho^{5}_{\Lambda_{3}^{c}},\ \rho^{5}_{\Lambda_{u}} \bigr)$.
We use this partitioning convention for all multipliers to maintain notational clarity and consistency throughout the analysis.

Based on the above setup, in $\eqref{0618}$, it holds that
\begin{equation}
    \begin{aligned}
        S_3
        &= \left(\rho^{1}\right)^{\top}\left[\mathbf{0}_{\left(I_{G_{1}},  \Psi_{2}\right)} \ \ I_{\left(I_{G_{1}},  \Psi_{3}\right)}\right]
        +\left(\rho^{4}\right)^{\top}\left[I_{\left(I_{H_{2}},  \Psi_{2}\right)} \ \  \mathbf{0}_{\left(I_{H_{2}},  \Psi_{3}\right)}\right]=
        \left[\begin{array}{cccc}
            \left(\rho^{4}\right)^{\top}&
            \left(\rho^{1}\right)^{\top}
        \end{array}\right]=0.\nonumber
    \end{aligned}
\end{equation}
By $\eqref{0115}$, it holds that
\begin{equation}
    \begin{aligned}
S_5
&=\left(\rho^{5}\right)^{\top}\left[\mathbf{0}_{\left(I_{G_{3}},  \Lambda_{1} \cup \Lambda_{2}\right)} \  \ I_{\left(I_{G_{3}},  \Lambda_{3} \cup \Lambda_{u}\right)}\right]+  \left(\rho^{6}\right)^{\top}\left[I_{\left(I_{G H_{3}},  \Lambda_{1}\right)} \  \ \mathbf{0}_{\left(I_{G H_{3}},  \Lambda_{2} \cup \Lambda_{3} \cup \Lambda_{u}\right)}\right] \nonumber\\
&+ \left(\rho^{11}\right)^{\top}\left[\mathbf{0}_{\left(I_{G H_{4}},  \Lambda_{1} \cup \Lambda_{2} \cup \Lambda_{3}^{+} \cup \Lambda_{u}\right) }\  \ I_{\left(I_{G H_{4}},  \Lambda_{3}^{c}\right)}\right]+\left(\rho^{12}\right)^{\top}\left[I_{\left(I_{H_{4}},  \Lambda_{1} \cup \Lambda_{2} \cup \Lambda_{3}^{+}\right) }\  \ \mathbf{0}_{\left(I_{H_{4}},  \Lambda_{3}^{c} \cup \Lambda_{u}\right)}\right]\nonumber\\
&=\left[
\begin{array}{ccccc}
\left(\rho^{6}\right)^{\top}+\left(\rho^{12}_{\Lambda_1}\right)^{\top}&
\left(\rho^{12}_{\Lambda_2}\right)^{\top}&
\left(\rho^{5}_{\Lambda_3^+}\right)^{\top}+\left(\rho^{12}_{\Lambda_3^+}\right)^{\top}&
\left(\rho^{5}_{\Lambda_3^c}\right)^{\top}+\left(\rho^{11}\right)^{\top}&
\left(\rho^{5}_{\Lambda_u}\right)^{\top}
\end{array}\right]=0.
\end{aligned}
\end{equation}
        It implies that  $\rho^{1}=0$, $\rho^{4}=0$, $\rho^{12}_{\Lambda_2}=0$ and $\rho^{5}_{\Lambda_u}=0$. Moreover, by \cite[Lemma 1]{ref1} and $\eqref{1103}$,  $\rho^{2}=0$ and $\rho^{3}=0$.
        Therefore, proving that the matrix $\Gamma
        $ demonstrates positive linear independence is reduced to examining the positive linear independence of the following matrix
        \begin{equation}\label{tildegamma}
			\widetilde{\Gamma}
			=
			\left[
			\begin{array}{ccccc}
				
				\mathbf{0}_{\left(I_{G_{3}},  L_{1}\right)} & (BB^{\top})_{\left(I_{G_{3}}, \cdot \right)} & \mathcal{I}_{\left(\Lambda_{3},  \cdot \right)} \\
				
				\mathbf{0}_{\left(I_{GH_{3}},  L_{1}\right)} & (BB^{\top})_{\left(I_{GH_{3}}, \cdot \right)} & \mathcal{I}_{\left(I_{GH_{3}},  \cdot \right)} \\
				
				\mathbf{0}_{\left(I_{GH_{3}},  L_{1}\right)} &  \mathcal{I}_{\left(I_{GH_{3}},  \cdot \right)} & \mathbf{0}_{\left(I_{GH_{3}},  L_{5}\right)} \\
				
				\mathbf{0}_{\left(I_{H_{3}},  L_{1}\right)} &  \mathcal{I}_{\left(I_{H_{3}},  \cdot \right)} & \mathbf{0}_{\left(I_{H_{3}},  L_{5}\right)} \\
				
				\mathbf{1}_{\left(I_{G_{4}},  L_{1}\right)} & \mathcal{I}_{\left(I_{G_{4}},  \cdot \right)} & \mathbf{0}_{\left(I_{G_{4}},  L_{5}\right)} \\
				
				\mathbf{1}_{\left(I_{G H_{4}},  L_{1}\right)} & \mathcal{I}_{\left(I_{GH_{4}},  \cdot \right)} & \mathbf{0}_{\left(I_{G H_{4}},  L_{5}\right)}\\
				
				\mathbf{0}_{\left(I_{G H_{4}},  L_{1}\right)} & \mathbf{0}_{\left(I_{G H_{4}},  L_{4}\right)} & \mathcal{I}_{\left(I_{GH_{4}},  \cdot \right)}\\
				
				\mathbf{0}_{\left(I_{H_{4}},  L_{1}\right)} & \mathbf{0}_{\left(I_{H_{4}},  L_{4}\right)} & \mathcal{I}_{\left(\Lambda_{1} \cup \Lambda_{3}^{+},  \cdot \right)}\\
			\end{array}
			\right].
		\end{equation}
        
        Note that $\widetilde{\Gamma}$ is the submatrix of $\widetilde{\Gamma}$ in \cite[Theorem 1]{li2023} corresponding to the second row block to the last row block. Therefore, using the same technique of proof in \cite[Theorem 1]{li2023}, if $\left(B B^{\top}\right)_{\left(\Lambda_{1} \cup \Lambda_{3}, \Lambda_{1} \cup \Lambda_{3}\right)}$ is positive definite at $v^*$, $\rho=0$. Then  the MPEC-MFCQ-R holds at $ v^* $.
        \end{proof}

        \section{Proof of Theorem \ref{LICQ_new}}\label{appendixd}
        \begin{proof}

		% \begin{proposition}
		 Taking into account of Proposition $\ref{pro1}$,  $\Gamma$ in $\eqref{gamma}$ takes the following form

        { 
{
\begin{equation}\label{gammaprop3}
\Gamma =
\begin{blockarray}{*{5}{c}c}
  \scriptstyle 1 & \scriptstyle 2 & \scriptstyle 3 & \scriptstyle 4 & \scriptstyle 5 & \\[-0.2ex]
\begin{block}{[ *{5}{c} ] c}
    \mathbf{0}_{(\Psi_3, L_1 )} & \mathbf{0}_{(\Psi_3, L_2 )} & \mathcal{I}_{\left(\Psi_3,\cdot\right)} & (AB^\top )_{(\Psi_3, \cdot)} & \mathbf{0}_{(\Psi_3, L_5 )} & \scriptstyle 1\\
  \mathbf{0}_{(\Psi_2, L_1 )} & \mathcal{I}_{\left(\Psi_2,\cdot\right)} & \mathbf{0}_{(\Psi_2,  L_{3})}  & \mathbf{0}_{(\Psi_{2},  L_{4}) } & \mathbf{0}_{(\Psi_{2},  L_{5})} & \scriptstyle 2\\
  \mathbf{0}_{(\Psi_3, L_1 )} &  \mathcal{I}_{\left(\Psi_3,\cdot\right)} & \mathbf{0}_{(\Psi_{3},  L_{3})  }&   \mathbf{0}_{(\Psi_{3},  L_{4}) }& \mathbf{0}_{(\Psi_{3},  L_{5})} & \scriptstyle 3\\
  \mathbf{0}_{(\Psi_2, L_1 )}& \mathbf{0}_{(\Psi_{2},  L_{2})}  & \mathcal{I}_{\left(\Psi_2,\cdot\right)}  &  \mathbf{0}_{(\Psi_{2},  L_{4}) } & \mathbf{0}_{(\Psi_{2},  L_{5}) }& \scriptstyle 4\\
  \mathbf{0}_{(\Lambda_{3}^+,  L_{1}) }& \mathbf{0}_{(\Lambda_{3}^+,  L_{2}) } &  \mathbf{0}_{(\Lambda_{3}^+,  L_3) } &  (B B^{\top})_{(\Lambda_{3}^+,  \cdot) } & \mathcal{I}_{\left(\Lambda_{3}^+,\cdot\right)} & \scriptstyle 5\\
  \mathbf{0}_{(\Lambda_{3}^c,  L_{1}) }& \mathbf{0}_{(\Lambda_{3}^c,  L_{2}) } &  \mathbf{0}_{(\Lambda_{3}^c,  L_3) } &  (B B^{\top})_{(\Lambda_{3}^c,  \cdot) } & \mathcal{I}_{\left(\Lambda_{3}^c,\cdot\right)} & \scriptstyle 6\\
  \mathbf{0}_{(\Lambda_{u},  L_{1}) } & \mathbf{0}_{(\Lambda_{u},  L_{2}) } &  \mathbf{0}_{(\Lambda_{u},  L_{3}) } &  (B B^{\top})_{(\Lambda_{u},  \cdot)} & \mathcal{I}_{\left(\Lambda_u,\cdot\right)} & \scriptstyle 7\\
  \mathbf{0}_{(\Lambda_{1},  L_{1}) } & \mathbf{0}_{(\Lambda_{1},  L_{2}) } &  \mathbf{0}_{(\Lambda_{1},  L_{3}) } & (B B^{\top})_{(\Lambda_{1},  \cdot) } & \mathcal{I}_{\left(\Lambda_1,\cdot\right)} & \scriptstyle 8\\
  \mathbf{0}_{(\Lambda_{1},  L_{1}) }&  \mathbf{0}_{(\Lambda_{1},  L_{2})}  & \mathbf{0}_{(\Lambda_{1},  L_{3}) } & \mathcal{I}_{\left(\Lambda_1,\cdot\right)}  & \mathbf{0}_{(\Lambda _{1},  L_{5})} & \scriptstyle 9\\
  \mathbf{0}_{(\Lambda_{2},  L_{1}) } & \mathbf{0}_{(\Lambda_{2},  L_{2}) } &  \mathbf{0}_{(\Lambda_{2},  L_{3}) } &  \mathcal{I}_{\left(\Lambda_2,\cdot\right)}  & \mathbf{0}_{(\Lambda_{2},  L_{5})} & \scriptstyle 10\\
  \mathbf{1}_{(\Lambda_{u},  L_{1}) } & \mathbf{0}_{(\Lambda_{u},  L_{2}) } &  \mathbf{0}_{(\Lambda_{u},  L_{3}) } &  \mathcal{I}_{\left(\Lambda_u,\cdot\right)}  & \mathbf{0}_{(\Lambda_{u},  L_{5})} & \scriptstyle 11\\
  \mathbf{1}_{(\Lambda_{3}^c,  L_1) } & \mathbf{0}_{(\Lambda_{3}^{c},  L_{2}) } & \mathbf{0}_{(\Lambda_{3}^{c},  L_{3}) } &  \mathcal{I}_{\left(\Lambda_{3}^c,\cdot\right)} & \mathbf{0}_{(\Lambda_{3}^{c},  L_{5})} & \scriptstyle 12\\
  \mathbf{0}_{(\Lambda_{3}^c,  L_{1}) }  & \mathbf{0}_{(\Lambda_{3}^{c},  L_{2}) } & \mathbf{0}_{(\Lambda_{3}^{c},  L_{3}) } & \mathbf{0}_{(\Lambda_{3}^{c},  L_{4}) } & \mathcal{I}_{\left(\Lambda_{3}^c,\cdot\right)} & \scriptstyle 13\\
  \mathbf{0}_{(\Lambda_{1},  L_{1}) } & \mathbf{0}_{(\Lambda_{1},  L_{2})} &  \mathbf{0}_{(\Lambda_{1},  L_{3}) } & \mathbf{0}_{(\Lambda_{1},  L_{4}) } & \mathcal{I}_{\left(\Lambda_1,\cdot\right)} & \scriptstyle 14\\
  \mathbf{0}_{(\Lambda_{2},  L_{1}) } &  \mathbf{0}_{(\Lambda_{2},  L_{2}) } & \mathbf{0}_{(\Lambda_{2},  L_{3})  }&  \mathbf{0}_{(\Lambda_{2},  L_{4}) }& \mathcal{I}_{\left(\Lambda_2,\cdot\right)} & \scriptstyle 15\\
  \mathbf{0}_{(\Lambda_{3}^+,  L_{1}) } & \mathbf{0}_{(\Lambda_{3}^{+},  L_{2}) } & \mathbf{0}_{(\Lambda_{3}^{+},  L_{3}) }  &  \mathbf{0}_{(\Lambda_{3}^{+},  L_{4}) } & \mathcal{I}_{\left(\Lambda_{3}^+,\cdot\right)}    & \scriptstyle 16\\
\end{block}
\end{blockarray}.
\end{equation}
}

		$\left(i\right)$ By the definition of $L_q$,  $q=1, ..., 5$,  there are $d$ columns in $\Gamma$. The number of rows is $d-1+\left | \Lambda_{1} \right | + \left | \Lambda_{3}^c \right |$ by the $\Gamma$ in $\eqref{gammaprop3}$. Thus, if $|\Lambda_{1}|+|\Lambda_{3}^c|>1$, then $\mathrm{rows}(\Gamma)>\mathrm{cols}(\Gamma)$. Hence the gradients of \eqref{mpec} are linearly dependent, and  the MPEC-LICQ fails at $v^*$.

    To show $\left(ii\right)$ and $\left(iii\right)$, 
    subtracting the thirteenth block from the sixth row block, subtracting the fourteenth block from the eighth row block, subtracting the sixteenth block from the fifth row block, we get the following matrix 

{
    \begin{equation}\label{gamma0new}
\Gamma^0 =
\begin{blockarray}{*{5}{c}c}
  \scriptstyle 1 & \scriptstyle 2 & \scriptstyle 3 & \scriptstyle 4 & \scriptstyle 5 & \\[-0.2ex]
\begin{block}{[ *{5}{c} ] c}
    \mathbf{0}_{(\Psi_3, L_1 )} & \mathbf{0}_{(\Psi_3, L_2 )} & \mathcal{I}_{\left(\Psi_3,\cdot\right)} & (AB^\top )_{(\Psi_3, \cdot)} & \mathbf{0}_{(\Psi_3, L_5 )} & \scriptstyle 1\\
  \mathbf{0}_{(\Psi_2, L_1 )} & \mathcal{I}_{\left(\Psi_2,\cdot\right)} & \mathbf{0}_{(\Psi_2,  L_{3})}  & \mathbf{0}_{(\Psi_{2},  L_{4}) } & \mathbf{0}_{(\Psi_{2},  L_{5})} & \scriptstyle 2\\
  \mathbf{0}_{(\Psi_3, L_1 )} &  \mathcal{I}_{\left(\Psi_3,\cdot\right)} & \mathbf{0}_{(\Psi_{3},  L_{3})  }&   \mathbf{0}_{(\Psi_{3},  L_{4}) }& \mathbf{0}_{(\Psi_{3},  L_{5})} & \scriptstyle 3\\
  \mathbf{0}_{(\Psi_2, L_1 )}& \mathbf{0}_{(\Psi_{2},  L_{2})}  &  \mathcal{I}_{\left(\Psi_2,\cdot\right)}  &  \mathbf{0}_{(\Psi_{2},  L_{4}) } & \mathbf{0}_{(\Psi_{2},  L_{5}) }& \scriptstyle 4\\
  \mathbf{0}_{(\Lambda_{3}^+,  L_{1}) }& \mathbf{0}_{(\Lambda_{3}^+,  L_{2}) } &  \mathbf{0}_{(\Lambda_{3}^+,  L_3) } &  (B B^{\top})_{(\Lambda_{3}^+,  \cdot) } & \mathbf{0}_{(\Lambda_{3}^+,  L_{5}) } & \scriptstyle 5\\
  \mathbf{0}_{(\Lambda_{3}^c,  L_{1}) }& \mathbf{0}_{(\Lambda_{3}^c,  L_{2}) } &  \mathbf{0}_{(\Lambda_{3}^c,  L_3) } &  (B B^{\top})_{(\Lambda_{3}^c,  \cdot) } & \mathbf{0}_{(\Lambda_{3}^c,  L_5) } & \scriptstyle 6\\
  \mathbf{0}_{(\Lambda_{u},  L_{1}) } & \mathbf{0}_{(\Lambda_{u},  L_{2}) } &  \mathbf{0}_{(\Lambda_{u},  L_{3}) } &  (B B^{\top})_{(\Lambda_{u},  \cdot)} & \mathcal{I}_{\left(\Lambda_u,\cdot\right)} & \scriptstyle 7\\
  \mathbf{0}_{(\Lambda_{1},  L_{1}) } & \mathbf{0}_{(\Lambda_{1},  L_{2}) } &  \mathbf{0}_{(\Lambda_{1},  L_{3}) } & (B B^{\top})_{(\Lambda_{1},  \cdot) } & \mathbf{0}_{(\Lambda_{1},  L_{5}) } & \scriptstyle 8\\
  \mathbf{0}_{(\Lambda_{1},  L_{1}) }&  \mathbf{0}_{(\Lambda_{1},  L_{2})}  & \mathbf{0}_{(\Lambda_{1},  L_{3}) } &  \mathcal{I}_{\left(\Lambda_1,\cdot\right)}  & \mathbf{0}_{(\Lambda _{1},  L_{5})} & \scriptstyle 9\\
  \mathbf{0}_{(\Lambda_{2},  L_{1}) } & \mathbf{0}_{(\Lambda_{2},  L_{2}) } &  \mathbf{0}_{(\Lambda_{2},  L_{3}) } &  \mathcal{I}_{\left(\Lambda_2,\cdot\right)} & \mathbf{0}_{(\Lambda_{2},  L_{5})} & \scriptstyle 10\\
  \mathbf{1}_{(\Lambda_{u},  L_{1}) } & \mathbf{0}_{(\Lambda_{u},  L_{2}) } &  \mathbf{0}_{(\Lambda_{u},  L_{3}) } &  \mathcal{I}_{\left(\Lambda_u,\cdot\right)} & \mathbf{0}_{(\Lambda_{u},  L_{5})} & \scriptstyle 11\\
  \mathbf{1}_{(\Lambda_{3}^c,  L_1) } & \mathbf{0}_{(\Lambda_{3}^{c},  L_{2}) } & \mathbf{0}_{(\Lambda_{3}^{c},  L_{3}) } &  \mathcal{I}_{\left(\Lambda_{3}^c,\cdot\right)}  & \mathbf{0}_{(\Lambda_{3}^{c},  L_{5})} & \scriptstyle 12\\
  \mathbf{0}_{(\Lambda_{3}^c,  L_{1}) }  & \mathbf{0}_{(\Lambda_{3}^{c},  L_{2}) } & \mathbf{0}_{(\Lambda_{3}^{c},  L_{3}) } & \mathbf{0}_{(\Lambda_{3}^{c},  L_{4}) } & \mathcal{I}_{\left(\Lambda_{3}^c,\cdot\right)} & \scriptstyle 13\\
  \mathbf{0}_{(\Lambda_{1},  L_{1}) } & \mathbf{0}_{(\Lambda_{1},  L_{2})} &  \mathbf{0}_{(\Lambda_{1},  L_{3}) } & \mathbf{0}_{(\Lambda_{1},  L_{4}) } & \mathcal{I}_{\left(\Lambda_1,\cdot\right)} & \scriptstyle 14\\
  \mathbf{0}_{(\Lambda_{2},  L_{1}) } &  \mathbf{0}_{(\Lambda_{2},  L_{2}) } & \mathbf{0}_{(\Lambda_{2},  L_{3})  }&  \mathbf{0}_{(\Lambda_{2},  L_{4}) }& \mathcal{I}_{\left(\Lambda_2,\cdot\right)} & \scriptstyle 15\\
  \mathbf{0}_{(\Lambda_{3}^+,  L_{1}) } & \mathbf{0}_{(\Lambda_{3}^{+},  L_{2}) } & \mathbf{0}_{(\Lambda_{3}^{+},  L_{3}) }  &  \mathbf{0}_{(\Lambda_{3}^{+},  L_{4}) } & \mathcal{I}_{\left(\Lambda_{3}^+,\cdot\right)}   & \scriptstyle 16\\
\end{block}
\end{blockarray}.
\end{equation}}

Starting from the block matrix $\Gamma^0$ in \eqref{gamma0new}, we apply a finite sequence of block row operations that combine the first, fifth, sixth, seventh, and eighth row blocks with the auxiliary row blocks indexed by $\Lambda_1$, $\Lambda_2$, $\Lambda_u$, and $\Lambda_3^c$. 
In the first row block, these operations eliminate the subblocks $(AB^\top)_{(\Psi_3,\Lambda_1)}$, $(AB^\top)_{(\Psi_3,\Lambda_2)}$, $(AB^\top)_{(\Psi_3,\Lambda_u)}$, and $(AB^\top)_{(\Psi_3,\Lambda_3^c)}$, while in the row blocks indexed by $\Lambda_3^+$, $\Lambda_3^c$, $\Lambda_u$, and $\Lambda_1$ (rows $5$–$8$) they eliminate the corresponding subblocks $(BB^\top)_{(\cdot,\Lambda_1)}$, $(BB^\top)_{(\cdot,\Lambda_2)}$, $(BB^\top)_{(\cdot,\Lambda_u)}$, and $(BB^\top)_{(\cdot,\Lambda_3^c)}$. 
These operations only modify the first column blocks in the affected rows, and the resulting entries are collected into the vectors $a_1,\dots,a_5$, while the overall rank of the matrix is preserved. 
After performing all these row operations, we reach the following matrix

{
\begin{equation}\label{gammahatnew}
\widehat{\Gamma} =
\begin{blockarray}{*{5}{c}c}
  \scriptstyle 1 & \scriptstyle 2 & \scriptstyle 3 & \scriptstyle 4 & \scriptstyle 5 & \\[-0.2ex]
\begin{block}{[ *{5}{c} ] c}
    a_1 & \mathbf{0}_{(\Psi_3, L_2 )} & \mathcal{I}_{\left(\Psi_3,\cdot\right)}  & Q_1 & \mathbf{0}_{(\Psi_3, L_5 )} & \scriptstyle 1\\
  \mathbf{0}_{(\Psi_2, L_1 )} & \mathcal{I}_{\left(\Psi_2,\cdot\right)}  & \mathbf{0}_{(\Psi_2,  L_{3})}  & \mathbf{0}_{(\Psi_{2},  L_{4}) } & \mathbf{0}_{(\Psi_{2},  L_{5})} & \scriptstyle 2\\
  \mathbf{0}_{(\Psi_3, L_1 )} &  \mathcal{I}_{\left(\Psi_3,\cdot\right)}  & \mathbf{0}_{(\Psi_{3},  L_{3})  }&   \mathbf{0}_{(\Psi_{3},  L_{4}) }& \mathbf{0}_{(\Psi_{3},  L_{5})} & \scriptstyle 3\\
  \mathbf{0}_{(\Psi_2, L_1 )}& \mathbf{0}_{(\Psi_{2},  L_{2})}  &  \mathcal{I}_{\left(\Psi_2,\cdot\right)}  &  \mathbf{0}_{(\Psi_{2},  L_{4}) } & \mathbf{0}_{(\Psi_{2},  L_{5}) }& \scriptstyle 4\\
  a_2& \mathbf{0}_{(\Lambda_{3}^+,  L_{2}) } &  \mathbf{0}_{(\Lambda_{3}^+,  L_3) } &  Q_2 & \mathbf{0}_{(\Lambda_{3}^+,  L_{5}) } & \scriptstyle 5\\
  a_3& \mathbf{0}_{(\Lambda_{3}^c,  L_{2}) } &  \mathbf{0}_{(\Lambda_{3}^c,  L_3) } &  Q_3 & \mathbf{0}_{(\Lambda_{3}^c,  L_5) } & \scriptstyle 6\\
  a_4 & \mathbf{0}_{(\Lambda_{u},  L_{2}) } &  \mathbf{0}_{(\Lambda_{u},  L_{3}) } &  Q_4 & \mathcal{I}_{\left(\Lambda_u,\cdot\right)} & \scriptstyle 7\\
  a_5 & \mathbf{0}_{(\Lambda_{1},  L_{2}) } &  \mathbf{0}_{(\Lambda_{1},  L_{3}) } & Q_5 & \mathbf{0}_{(\Lambda_{1},  L_{5}) } & \scriptstyle 8\\
  \mathbf{0}_{(\Lambda_{1},  L_{1}) }&  \mathbf{0}_{(\Lambda_{1},  L_{2})}  & \mathbf{0}_{(\Lambda_{1},  L_{3}) } &  \mathcal{I}_{\left(\Lambda_1,\cdot\right)}  & \mathbf{0}_{(\Lambda _{1},  L_{5})} & \scriptstyle 9\\
  \mathbf{0}_{(\Lambda_{2},  L_{1}) } & \mathbf{0}_{(\Lambda_{2},  L_{2}) } &  \mathbf{0}_{(\Lambda_{2},  L_{3}) } &  \mathcal{I}_{\left(\Lambda_2,\cdot\right)} & \mathbf{0}_{(\Lambda_{2},  L_{5})} & \scriptstyle 10\\
  \mathbf{1}_{(\Lambda_{u},  L_{1}) } & \mathbf{0}_{(\Lambda_{u},  L_{2}) } &  \mathbf{0}_{(\Lambda_{u},  L_{3}) } &  \mathcal{I}_{\left(\Lambda_u,\cdot\right)}  & \mathbf{0}_{(\Lambda_{u},  L_{5})} & \scriptstyle 11\\
  \mathbf{1}_{(\Lambda_{3}^c,  L_1) } & \mathbf{0}_{(\Lambda_{3}^{c},  L_{2}) } & \mathbf{0}_{(\Lambda_{3}^{c},  L_{3}) } &  \mathcal{I}_{\left(\Lambda_{3}^{c},\cdot\right)}  & \mathbf{0}_{(\Lambda_{3}^{c},  L_{5})} & \scriptstyle 12\\
  \mathbf{0}_{(\Lambda_{3}^c,  L_{1}) }  & \mathbf{0}_{(\Lambda_{3}^{c},  L_{2}) } & \mathbf{0}_{(\Lambda_{3}^{c},  L_{3}) } & \mathbf{0}_{(\Lambda_{3}^{c},  L_{4}) } & \mathcal{I}_{\left(\Lambda_{3}^{c},\cdot\right)} & \scriptstyle 13\\
  \mathbf{0}_{(\Lambda_{1},  L_{1}) } & \mathbf{0}_{(\Lambda_{1},  L_{2})} &  \mathbf{0}_{(\Lambda_{1},  L_{3}) } & \mathbf{0}_{(\Lambda_{1},  L_{4}) } & \mathcal{I}_{\left(\Lambda_1,\cdot\right)} & \scriptstyle 14\\
  \mathbf{0}_{(\Lambda_{2},  L_{1}) } &  \mathbf{0}_{(\Lambda_{2},  L_{2}) } & \mathbf{0}_{(\Lambda_{2},  L_{3})  }&  \mathbf{0}_{(\Lambda_{2},  L_{4}) }& \mathcal{I}_{\left(\Lambda_2,\cdot\right)} & \scriptstyle 15\\
  \mathbf{0}_{(\Lambda_{3}^+,  L_{1}) } & \mathbf{0}_{(\Lambda_{3}^{+},  L_{2}) } & \mathbf{0}_{(\Lambda_{3}^{+},  L_{3}) }  &  \mathbf{0}_{(\Lambda_{3}^{+},  L_{4}) } & \mathcal{I}_{\left(\Lambda_{3}^{+},\cdot\right)}   & \scriptstyle 16\\
\end{block}
\end{blockarray},
\end{equation}
}
where 
{\small
\begin{align*}
            a_1&=(AB^{\top})_{(\Psi_3, \Lambda_{3}^c\cup\Lambda_{u})}\mathbf{1}_{|\Lambda_{3}^c\cup\Lambda_{u}|},\ 
			a_2=(BB^{\top})_{(\Lambda_{3}^+, \Lambda_{3}^c\cup\Lambda_{u})}\mathbf{1}_{|\Lambda_{3}^c\cup\Lambda_{u}|},\\   
			a_3&=(BB^{\top})_{(\Lambda_{3}^c, \Lambda_{3}^c\cup\Lambda_{u})}\mathbf{1}_{|\Lambda_{3}^c\cup\Lambda_{u}|} ,\
            a_4=(BB^{\top})_{(\Lambda_u, \Lambda_{3}^c\cup\Lambda_{u})}\mathbf{1}_{|\Lambda_{3}^c\cup\Lambda_{u}|},\\   
			a_5&=(BB^{\top})_{(\Lambda_1, \Lambda_{3}^c\cup\Lambda_{u})}\mathbf{1}_{|\Lambda_{3}^c\cup\Lambda_{u}|} ,\
            Q_1=\left[\begin{array}{ll}
            (AB^\top )_{\left(\Psi_3, \Lambda_{3}^{+}\right)} & 0_{\left(\Psi_3,\Lambda_{1} \cup \Lambda_{2} \cup \Lambda_{3}^{c} \cup \Lambda_{u}\right)}
            \end{array}\right],\\ 
            Q_2 &=\left[\begin{array}{ll}
            (BB^\top )_{\left(\Lambda_{3}^{+}, \Lambda_{3}^{+}\right)} & \mathbf{0}_{\left(\Lambda_{3}^{+},\Lambda_{1} \cup \Lambda_{2} \cup \Lambda_{3}^{c} \cup \Lambda_{u}\right)}
            \end{array}\right],\
            Q_3=\left[\begin{array}{ll}
            (BB^\top )_{\left(\Lambda_{3}^{c}, \Lambda_{3}^{+}\right)} & \mathbf{0}_{\left(\Lambda_{3}^{c},\Lambda_{1} \cup \Lambda_{2} \cup \Lambda_{3}^{c} \cup \Lambda_{u}\right)}
            \end{array}\right],\\ 
            Q_4&=\left[\begin{array}{ll}
            (BB^\top )_{\left(\Lambda_{u}, \Lambda_{3}^{+}\right)} & \mathbf{0}_{\left(\Lambda_{u},\Lambda_{1} \cup \Lambda_{2} \cup \Lambda_{3}^{c} \cup \Lambda_{u}\right)}
            \end{array}\right],\
            Q_5=\left[\begin{array}{ll}
            (BB^\top )_{\left(\Lambda_{1}, \Lambda_{3}^{+}\right)} & \mathbf{0}_{\left(\Lambda_{1},\Lambda_{1} \cup \Lambda_{2} \cup \Lambda_{3}^{c} \cup \Lambda_{u}\right)}
            \end{array}\right].
		\end{align*}
        }
        After suitable permutations of rows and columns, the resulting matrix $\widetilde{\Gamma}$ has the same row rank as $\widehat{\Gamma}$.

{
\begin{equation}\label{Gammatilde}
\widetilde{\Gamma} =
\begin{blockarray}{*{5}{c}c}
  \scriptstyle 4 & \scriptstyle 3 & \scriptstyle 1 & \scriptstyle 2 & \scriptstyle 5 & \\[-0.2ex]
\begin{block}{[ *{5}{c} ] c}
    \mathcal{I}_{\left(\Lambda_1,\cdot\right)} & \mathbf{0}_{(\Lambda_{1}, L_{3})} & \mathbf{0}_{(\Lambda_{1}, L_{1})} & \mathbf{0}_{(\Lambda_{1}, L_{2})} & \mathbf{0}_{(\Lambda_{1}, L_{5})} & \scriptstyle 9\\
  \mathcal{I}_{\left(\Lambda_2,\cdot\right)} & \mathbf{0}_{(\Lambda_{2}, L_{3})} & \mathbf{0}_{(\Lambda_{2}, L_{1})} & \mathbf{0}_{(\Lambda_{2}, L_{2})} & \mathbf{0}_{(\Lambda_{2}, L_{5})} & \scriptstyle 10\\
  \mathcal{I}_{\left(\Lambda_u,\cdot\right)} & \mathbf{0}_{(\Lambda_{u}, L_{3})} & \mathbf{1}_{(\Lambda_{u}, L_{1})} & \mathbf{0}_{(\Lambda_{u}, L_{2})} & \mathbf{0}_{(\Lambda_{u}, L_{5})} & \scriptstyle 11\\
  \mathcal{I}_{\left(\Lambda_3^c,\cdot\right)} & \mathbf{0}_{(\Lambda_{3}^{c}, L_{3})} & \mathbf{1}_{(\Lambda_{3}^{c}, L_{1})} & \mathbf{0}_{(\Lambda_{3}^{c}, L_{2})} & \mathbf{0}_{(\Lambda_{3}^{c}, L_{5})} & \scriptstyle 12\\
  {\color{green}Q_2}
  & \mathbf{0}_{(\Lambda_{3}^{+}, L_{3})} & {\color{green}a_{2}} & \mathbf{0}_{(\Lambda_{3}^{+}, L_{2})} & \mathbf{0}_{(\Lambda_{3}^{+}, L_{5})} & \scriptstyle 5\\
  {\color{green}Q_3}
  & \mathbf{0}_{(\Lambda_{3}^{c}, L_{3})} & {\color{green}a_{3}} & \mathbf{0}_{(\Lambda_{3}^{c}, L_{2})} & \mathbf{0}_{(\Lambda_{3}^{c}, L_{5})} & \scriptstyle 6\\
  Q_4
  & \mathbf{0}_{(\Lambda_{u}, L_{3})} & a_{4} & \mathbf{0}_{(\Lambda_{u}, L_{2})} & \mathcal{I}_{\left(\Lambda_u,\cdot\right)} & \scriptstyle 7\\
  {\color{green}Q_5}
  & \mathbf{0}_{(\Lambda_{1}, L_{3})} & {\color{green}a_{5}} & \mathbf{0}_{(\Lambda_{1}, L_{2})} & \mathbf{0}_{(\Lambda_{1}, L_{5})} & \scriptstyle 8\\
  Q_1
  &\mathcal{I}_{\left(\Psi_{2},\cdot\right)} & a_{1} & \mathbf{0}_{(\Psi_{3}, L_{2})} & \mathbf{0}_{(\Psi_{3}, L_{5})} & \scriptstyle 1\\
  \mathbf{0}_{(\Psi_{2}, L_{4})} & \mathbf{0}_{(\Psi_{2}, L_{3})} & \mathbf{0}_{(\Psi_{2}, L_{1})} & \mathcal{I}_{\left(\Psi_{3},\cdot\right)} & \mathbf{0}_{(\Psi_{2}, L_{5})} & \scriptstyle 2\\
  \mathbf{0}_{(\Psi_{3}, L_{4})} & \mathbf{0}_{(\Psi_{3}, L_{3})} & \mathbf{0}_{(\Psi_{3}, L_{1})} & \mathcal{I}_{\left(\Psi_{3},\cdot\right)} & \mathbf{0}_{(\Psi_{3}, L_{5})} & \scriptstyle 3\\
  \mathbf{0}_{(\Psi_{2}, L_{4})} & \mathcal{I}_{\left(\Psi_{2},\cdot\right)} & \mathbf{0}_{(\Psi_{2}, L_{1})} & \mathbf{0}_{(\Psi_{2}, L_{2})} & \mathbf{0}_{(\Psi_{2}, L_{5})} & \scriptstyle 4\\
  \mathbf{0}_{(\Lambda_{3}^{c}, L_{4})} & \mathbf{0}_{(\Lambda_{3}^{c}, L_{3})} & \mathbf{0}_{(\Lambda_{3}^{c}, L_{1})} & \mathbf{0}_{(\Lambda_{3}^{c}, L_{2})} & \mathcal{I}_{\left(\Lambda_3^c,\cdot\right)} & \scriptstyle 13\\
  \mathbf{0}_{(\Lambda_{1}, L_{4})} & \mathbf{0}_{(\Lambda_{1}, L_{3})} & \mathbf{0}_{(\Lambda_{1}, L_{1})} & \mathbf{0}_{(\Lambda_{1}, L_{2})} & \mathcal{I}_{\left(\Lambda_1,\cdot\right)} & \scriptstyle 14\\
  \mathbf{0}_{(\Lambda_{2}, L_{4})} & \mathbf{0}_{(\Lambda_{2}, L_{3})} & \mathbf{0}_{(\Lambda_{2}, L_{1})} & \mathbf{0}_{(\Lambda_{2}, L_{2})} & \mathcal{I}_{\left(\Lambda_2,\cdot\right)} & \scriptstyle 15\\
  \mathbf{0}_{(\Lambda_{3}^{+}, L_{4})} & \mathbf{0}_{(\Lambda_{3}^{+}, L_{3})} & \mathbf{0}_{(\Lambda_{3}^{+}, L_{1})} & \mathbf{0}_{(\Lambda_{3}^{+}, L_{2})} & \mathcal{I}_{\left(\Lambda_3^+,\cdot\right)}   & \scriptstyle 16\\
\end{block}
\end{blockarray}.
\end{equation}
}

Reorder the $L_4$ subcolumns as $(\Lambda_{1},\Lambda_{2},\Lambda_{3}^{c},\Lambda_u,\Lambda_{3}^{+})$.
Use rows $9$–$12$ as pivots and perform block Gaussian elimination.
These rank–preserving steps zero all nonpivot entries in the first four $L_4$ subcolumns.
Hence every $Q$–row has support in $L_4$ only on $\Lambda_{3}^{+}$.
Hence $\widetilde{\Gamma}$ has full row rank if and only if $\Gamma_{\mathrm{sub}}$ has full row rank, where
\[
\Gamma_{\mathrm{sub}}^{Q}=
\begin{bmatrix}
(BB^\top)_{(\Lambda_{3}^{+},\Lambda_{3}^{+})} & a_{2}  \\
(BB^\top)_{(\Lambda_{3}^{c},\Lambda_{3}^{+})} & a_{3}  \\
(BB^\top)_{(\Lambda_{1},\Lambda_{3}^{+})}     & a_{5} 

\end{bmatrix}.
\]

Since $(BB^\top)_{(\Lambda_{3}^{+},\Lambda_{3}^{+})}\succ0$, pivot on this block and perform block Gaussian elimination to clear the first block column.
Elementary row operations then yield the rank–equivalent form
\[
\begin{bmatrix}
(BB^\top)_{(\Lambda_{3}^{+},\Lambda_{3}^{+})} & a_{2}  \\
0 & \tilde a_{3}  \\
0 & \tilde a_{5} 
\end{bmatrix},
\qquad
\begin{cases}
\tilde a_{3}=a_3-(BB^\top)_{(\Lambda_{3}^{c},\Lambda_{3}^{+})}(BB^\top)_{(\Lambda_{3}^{+},\Lambda_{3}^{+})}^{-1}a_2,\\
\tilde a_{5}=a_5-(BB^\top)_{(\Lambda_{1},\Lambda_{3}^{+})}(BB^\top)_{(\Lambda_{3}^{+},\Lambda_{3}^{+})}^{-1}a_2.
\end{cases}
\]

Thus,
\[
\operatorname{rank}\big(\Gamma_{\mathrm{sub}}^{Q}\big)
=|\Lambda_{3}^{+}|
+\operatorname{rank}\!\begin{bmatrix}\tilde a_{3}\\ \tilde a_{5}\end{bmatrix}.
\]
In particular, $\Gamma_{\mathrm{sub}}^{Q}$ has full row rank if and only if
$\begin{bmatrix}\tilde a_{3}\\ \tilde a_{5}\end{bmatrix}$ has full row rank. Note that $\begin{bmatrix}\tilde a_{3}\\ \tilde a_{5}\end{bmatrix}$ is a single–column matrix.
Hence full row rank is possible only when the total number of its rows is at most one; equivalently,
$|\Lambda_{3}^{c}|+|\Lambda_{1}|\in\{0,1\}.$
That is, either $|\Lambda_{3}^{c}|=1,\,|\Lambda_{1}|=0$, or $|\Lambda_{3}^{c}|=0,\,|\Lambda_{1}|=1$, or both sets are empty.
In the one–row cases, full row rank further reduces to
$\begin{bmatrix}\tilde a_{3}\\ \tilde a_{5}\end{bmatrix}\neq 0$. Recall the definition of $\tilde a_{3}$, $\tilde a_{5}$ and $\Gamma_{\mathrm{sub}}$, we get the result.
}
\end{proof}
	
	\iffalse

	\fi
	

\begin{thebibliography}{}\label{references}
    \bibitem{Bard} Bard JF. Practical bilevel optimization: algorithms and applications. Springer Science \& Business Media, 2013.

    { 
     \bibitem{SDempe1} Dempe S. Bilevel optimization: Theory, algorithms, applications and a bibliography, Bilevel Optimization. Springer. 2020;  581–672.
     }

     \bibitem{DempeZemkohoBook} Dempe S, Zemkoho A. Bilevel optimization. In: Springer Optimization and Its Applications. Vol. 161. Springer; 2020.

    \bibitem{SDempe2} Dempe S,  Kalashnikov V,  Pérez-Valdés GA,  et al. Bilevel programming problems.  Energy Systems. Springer, Berlin, 2015; 10(978-3): 53-56.

    \bibitem{Sinha} Sinha A,  Malo P,  Deb K. A review on bilevel optimization: From classical to evolutionary approaches and applications. IEEE transactions on evolutionary computation. 2017; 22(2): 276-295.

    \bibitem{Wogrin} Wogrin S,  Pineda S,  Tejada-Arango DA. Applications of bilevel optimization in energy and electricity markets. Bilevel Optimization: Advances and Next Challenges.  2020; 139-168.


    \bibitem{hyper2} Okuno T,  Takeda A,  Kawana A,  et al. On $\ell_p$-hyperparameter learning via bilevel nonsmooth optimization. Journal of Machine Learning Research.  2021;  22(245): 1-47.



    \bibitem{architecture2}Lian D,  Zheng Y,  Xu Y,  et al. Towards fast adaptation of neural architectures with meta learning. International Conference on Learning Representations. 2020.



    \bibitem{rein2} Yang Z,  Chen Y,  Hong M,  et al. Provably global convergence of actor-critic: A case for linear quadratic regulator with ergodic cost. Advances in neural information processing systems.  2019;  32.

    

 

    \bibitem{image2} Ma L, Jin D, An N, et al. Bilevel fast scene adaptation for low-light image enhancement. International Journal of Computer Vision. 2023; 1-19.

    \bibitem{SDempe3} Dempe S,  Franke S. On the solution of convex bilevel optimization problems. Computational Optimization and Applications. 2016; 63: 685-703.


    \bibitem{DempeZemkoho2013} Dempe S, Zemkoho A. The bilevel programming problem: reformulations, constraint qualifications and optimality conditions. Mathematical Programming. 2013; 138: 447-473.

    \bibitem{DempeZemkoho2012} Dempe S, Zemkoho A. On the Karush–Kuhn–Tucker reformulation of the bilevel optimization problem. Nonlinear Analysis: Theory, Methods \& Applications. 2012; 75(3): 1202-1218.

     \bibitem{LuoZQMPEC} Luo ZQ,  Pang JS,  Ralph D. Mathematical programs with equilibrium constraints. Cambridge University Press.  1996.


    \bibitem{MLF1} Flegel ML,  Kanzow C. Abadie-type constraint qualification for mathematical programs with equilibrium constraints. Journal of Optimization Theory and Applications.  2005; 124(3): 595-614.

    \bibitem{ramos} Ramos A. Mathematical programs with equilibrium constraints: a sequential optimality condition, new constraint qualifications and algorithmic consequences. Optimization Methods and Software. 2021; 36(1): 45-81.
    
     \bibitem{Pang1999} Pang JS and Fukushima M. Complementarity constraint qualifications and simplified B-stationarity conditions for mathematical programs with equilibrium constraints. Computational Optimization and Applications. 1999; 13: 111-136.

    \bibitem{Scheel2000} H. Scheel and S. Scholtes, Mathematical programs with complementarity constraints: Stationarity, optimality, and sensitivity. Mathematics of Operations Research, 2000, 25(1): 1-22.

    \bibitem{Ralph04} Ralph* D,  Wright SJ. Some properties of regularization and penalization schemes for MPECs. Optimization Methods and Software. 2004; 19(5): 527-556.

     \bibitem{JaneYe05} Jane JY. Necessary and sufficient optimality conditions for mathematical programs with equilibrium constraints. Journal of Mathematical Analysis and Applications. 2005; 307(1): 350-369.
     
    
    \bibitem{MFL2} Flegel ML,  Kanzow C. A direct proof for M-stationarity under MPEC-GCQ for mathematical programs with equilibrium constraints. Optimization with Multivalued Mappings: Theory, Applications, and Algorithms. Boston, MA: Springer US, 2006; 111-122.


    

    \bibitem{Henrion} Henrion R,  Outrata JV. Calmness of constraint systems with applications. Mathematical Programming. 2005; 104(2): 437-464.

    { 

\bibitem{CL13} Chieu NH, Lee GM. A relaxed constant positive linear dependence constraint qualification for mathematical programs with equilibrium constraints. Journal of Optimization Theory and Applications. 2013; 158(1): 11-32.

\bibitem{Alberto19} Ramos A. Two new weak constraint qualifications for mathematical programs with equilibrium constraints and applications. Journal of Optimization Theory and Applications. 2019; 183(2): 566-591.

\bibitem{GZL14} Guo L, Zhang J, and Lin GH. New results on constraint qualifications for nonlinear extremum problems and extensions. Journal of Optimization Theory and Applications. 2014; 163: 737-754.




\bibitem{GLY13} Guo L, Lin GH, and Ye JJ. Second-order optimality conditions
for mathematical programs with equilibrium constraints. Journal of Optimization Theory
and Applications. 2013; 158: 33-64.

\bibitem{GL13} Guo L, Lin GH. Notes on some constraint qualifications for mathematical programs with equilibrium constraints. Journal of Optimization Theory and Applications. 2013; 156: 600-616.

% \bibitem{Alberto21} Ramos A. Mathematical programs with equilibrium constraints: a sequential optimality condition, new constraint qualifications and algorithmic consequences. Optimization Methods and Software. 2021; 36(1): 45-81.
}

        \bibitem{Bennett2006} Bennett KP, Hu J, Ji X, et al. Model selection via bilevel optimization. The 2006 IEEE International Joint Conference on Neural Network Proceedings. IEEE, 2006; 1922-1929.
		
   
		
    \bibitem{Kunapuli2008a} Kunapuli G, Bennett KP, Hu J, et al. Classification model selection via bilevel programming. Optimization Methods \& Software. 2008; 23(4): 475-489.
		
    \bibitem{Kunapuli2008b}Kunapuli G, Bennett KP, Hu J, et al. Bilevel model selection for support vector machines. Data Mining and Mathematical Programming. 2008; 45: 129–158.

     \bibitem{li2022}  Li Z,  Qian Y,  Li Q. A unified framework and a case study for hyperparameter selection in machine learning via bilevel optimization. 5th International Conference on Data Science and Information Technology (DSIT), 2022, pp. 1-8.

     \bibitem{ref1} Li Q,  Li Z,  Zemkoho A. Bilevel hyperparameter optimization for support vector classification: theoretical analysis and a solution method. Mathematical Methods of Operations Research.  2022;  96(3): 315-350.
		
    \bibitem{li2023}   Qian Y, Li Q, Zemkoho A. Global relaxation-based LP-Newton method for multiple hyperparameter selection in support vector classification with feature selection. arXiv preprint arXiv:2312.10848, 2023.

    \bibitem{coniglio2023}  Coniglio S, Dunn S, Li Q, et al. Bilevel hyperparameter optimization for nonlinear support vector machines. Optimization Online, 2023, pp. 1-78. 

    \bibitem{samuel}  Ward S, Zemkoho A, Ahipasaoglu S. Mathematical programs with complementarity constraints and application to hyperparameter tuning for nonlinear support vector machines, arXiv preprint arXiv:2504.13006, 2025.


		
    \bibitem{yx24} Wang Y, Li Q. A fast smoothing Newton method for bilevel hyperparameter optimization for SVC with Logistic loss. Optimization, 2024, DOI: 10.1080/02331934.2024.2394612.

    \bibitem{Forsgren02} Forsgren A, Gill P E, Wright M H. Interior methods for nonlinear optimization. SIAM review. 2002; 44(4): 525-597.

    \bibitem{Hoheisel2013} Hoheisel T,  Kanzow C,  Schwartz A. Theoretical and numerical comparison of relaxation methods for mathematical programs with complementarity constraints. Mathematical Programming. 2013; 137: 257-288.

    \bibitem{Abadie1967}
    Abadie J. On the Kuhn-Tucker theorem. No. ORC6518. 1966.

    \bibitem{Scholtes2001} Scholtes S. Convergence properties of a regularization scheme for mathematical programs with complementarity constraints. SIAM Journal on Optimization. 2001; 11(4): 918-936.


    \bibitem{TCA12} Hoheisel T, Kanzow C, Schwartz A. Convergence of a local regularization approach for mathematical programmes with complementarity or vanishing constraints. Optimization Methods and Software. 2012; 27(3): 483-512.

    \bibitem{MFL5} Flegel ML,  Kanzow C,  Outrata JV. Optimality conditions for disjunctive programs with application to mathematical programs with equilibrium constraints. Set-Valued Analysis. 2007;  15(2): 139-162.

    \bibitem{Robinson09} Robinson S M. Some continuity properties of polyhedral multifunctions. Mathematical Programming at Oberwolfach. Berlin, Heidelberg: Springer Berlin Heidelberg. 2009; 206-214.

     \bibitem{Uderzo21} Uderzo A. On the quantitative solution stability of parameterized set-valued inclusions. Set-Valued and Variational Analysis. 2021; 29(2): 425-451.

    \bibitem{Henrion18} Henrion R. Calmness as a constraint qualification for M-stationarity conditions in MPECs. Generalized Nash Equilibrium Problems, Bilevel Programming and MPEC. Singapore: Springer Singapore, 2018: 21-41.

         \bibitem{MFL3} Flegel ML,  Kanzow C. On the Guignard constraint qualification for mathematical programs with equilibrium constraints. Optimization. 2005; 54(6): 517-534.

         \bibitem{Kanzow10} Kanzow C, Schwartz A. Mathematical programs with equilibrium constraints: Enhanced Fritz John conditions,  new constraint qualifications,  and improved exact penalty results. SIAM Journal on Optimization. 2010; 20(5): 2730-2753.

         \bibitem{Rockafellar98} Rockafellar RT, Wets RJB. Variational analysis. Berlin, Heidelberg: Springer Berlin Heidelberg, 1998.

         { 
    \bibitem{KS13}Kanzow C, Schwartz A. A new regularization method for mathematical programs with complementarity constraints with strong convergence properties. SIAM Journal on Optimization. 2013; 23(2): 770-798.
}

{ 
    \bibitem{KDB09}Kadrani A, Dussault JP, Benchakroun A. A new regularization scheme for mathematical programs with complementarity constraints. SIAM Journal on Optimization. 2009; 20(1): 78-103.

}

\bibitem{Mangasarian1994} Mangasarian OL. Misclassification minimization. Journal of Global Optimization. 1994; 5(4): 309-323.

{ 
\bibitem{BV04} Boyd S, Vandenberghe L. Convex Optimization. Cambridge University Press. 2004.

\bibitem{LuenYe08} Luenberger D G, Ye Y. Linear and Nonlinear Programming (3rd ed.). Springer. 2008.
}
    

    \end{thebibliography}
\end{document}